\documentclass{article}
\usepackage[utf8]{inputenc}
\usepackage{geometry}
\usepackage{amssymb}
\usepackage{amsthm,epsf,graphics,wrapfig,picinpar,fancyhdr,caption2,amsmath}
\usepackage{ bbold }
\usepackage{ mathrsfs }
\usepackage{xcolor}
\usepackage{graphicx}
\usepackage{enumerate}
\usepackage{bigints}
 \numberwithin{equation}{section}
\theoremstyle{definition}
\newtheorem{Theorem}{Theorem}
\newtheorem{Proposition}{Proposition}
\newtheorem{Definition}{Definition}
\newtheorem{Remark}{Remark}

\newtheorem{Lemma}{Lemma}
\newtheorem{Example}{Example}
\numberwithin{Lemma}{section}
\numberwithin{Theorem}{section}
\numberwithin{Proposition}{section}
\numberwithin{Definition}{section}
\numberwithin{Remark}{section}

\title{Three-dimensional Gaussian fluctuations of  non-commutative random surface growth with a reflecting wall}
\date{}
\author{ Zhengye Zhou\footnote{Texas A$\&$M University, College Station, TX, United States of America.  Email: zyzhou@tamu.edu}}
\begin{document}
\maketitle
\begin{abstract}
    We consider the multi-time correlation and covariance structure of a random surface growth with a wall introduced in \cite{borodin2011random}. It is shown that the correlation functions associated with the model along space-like paths have determinantal structure, which yields the convergence of height fluctuations to that of a Gaussian free field.  We also construct a continuous-time non-commutative random walk on $U(\mathfrak{so}_{N+1})$, which matches the random surface growth when restricting to the Gelfand-Tsetlin subalgebra of $U(\mathfrak{so}_{N+1})$. As an application, we  prove the convergence of moments to an explicit Gaussian free field and get the  covariance functions of the associated random point process along both the space-like paths and time-like paths. In particular, it does not match the three-dimensional Gaussian field from spectra of  overlapping
stochastic Wishart matrices in \cite{kuan2021threedimensional} even along the space-like paths.
\end{abstract}

\section{Introduction}
As an approach to study the Anisotropic Kardar–Parisi–Zhang (AKPZ) equation,  many models in the AKPZ universality class were  studied over the past decades (e.g. \cite{BCT-Sto,PS,TF2+1}). There has been lots of progress in understanding large time asymptotics of driven interacting particle systems on the $2+1$ dimensions random growth models in the AKPZ universality class (e.g. \cite{BCF-Aniso,BCT-Sto,BF-Anisotropic,2013}). For example, in \cite{2013}, the authors constructed a class of two-dimensional random surface growth models, which can be interpreted as  random point processes. It was shown that along space-like paths, the point processes are determinantal. The authors also established Gaussian fluctuations of one specific growing random surface along space-like paths by computing the correlation kernel and taking asymptotics. Later in \cite{kuan2014threedimensional},  a continuous-time non-commutative random walk on   $U(\mathfrak{gl}_N)$ (the universal enveloping algebra of $\mathfrak{gl}_N$) was introduced, which matches the random surface growth model introduced in \cite{2013} when restricting to the Gelfand-Tsetlin subalgebra of $U(\mathfrak{gl}_N)$. As an application, the convergence to the Gaussian fluctuations along  time-like paths was proved, completing the entire three-dimensional Gaussian field. 
In particular, it was also shown that this three-dimensional Gaussian field matches the one for eigenvalues of stochastic Wigner matrices in \cite{AB-CLT}.

Another model of interest is the random surface growth with a reflecting wall constructed in \cite{borodin2011random}, which can be viewed as an one-parameter family of Plancherel
measures for the infinite-dimensional orthogonal group. It is believed that this model also belongs to the AKPZ universality class with certain boundary condition. Shown in Figure \ref{fig:2}, the random surface growth is equivalent to a particle process in the quarter plane.

\begin{figure}[h]
    \centering
    \includegraphics[width=0.2\columnwidth]{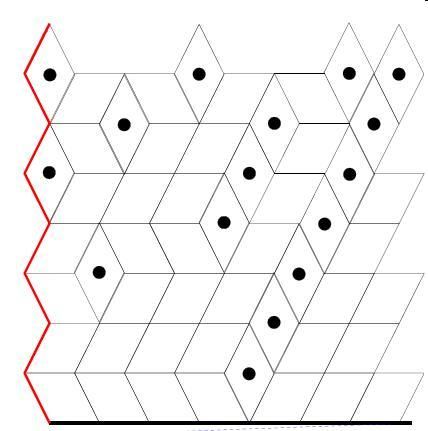}
    \caption{random surface growth with a wall}
    \label{fig:2}
\end{figure}

 In \cite{borodin2011random},   the determinantal formula for the correlation functions at any finite time moment was derived. However, the authors did not provide the correlation functions along space-like paths since the evolution of measure does not match the general L-ensembles (see e.g. \cite{BCF-Aniso,Borodin_2007,BRE-Eynard}). More specifically, this model has $n$ particles at both $(2n-1/2\pm 1/2)$-th levels, thus the the corresponding L-ensemble should  only have virtual variables for even levels.   The Gaussian free field fluctuations for the height functions at any fixed time moment were later proved in \cite{10.1214/EJP.v19-3732}.  In this paper, we first introduce a generalized L-ensemble, which implies the determinantal property and the formula for the correlation function.  We then apply this generalization to  extend the determinantal formula for the correlation functions to  multi-time moments. Later we introduce a continuous-time non-commutative random walk on the universal enveloping algebra of the Lie group $\mathfrak{so}_{N+1}$ (denoted $U(\mathfrak{so}_{N+1})$), which is an analog of the non-commutative random walks on $U(\mathfrak{gl}_N)$ (see e.g. \cite{B-Quantum,CD-Quantum,kuan2014threedimensional}). It can be proved that this random walk matches the random surface growth. With the help of the non-commutative random walk, we obtain the three-dimensional Gaussian fluctuations and the covariance structures of the random surface growth. It is natural to compare the three-dimensional Gaussian free fields from both interacting particle systems and the eigenvalue processes of random matrices that belong to the AKPZ universality class. For our model, an appropriate analog is the Wishart matrices, since the eigenvalues are always non–negative,
corresponding to the fact that particles are restricted to be to the right of the reflecting wall. However, the Gaussian fluctuation fields from this paper and \cite{kuan2021threedimensional}  are different along both the space-like paths and the time-like paths.

\textbf{Outline of the paper.}
In section 2, we review some facts about the algebras and the random surface growth. 
In section 3, we introduce a generalized L-ensemble to show that our model is a determinantal random point process along space-like paths. Moreover, we compute the multi-time correlation kernel which extends the existing formulas  in \cite{borodin2011random} for the single-time kernel. Later, we define a continuous-time non-commutative random walk on $U(\mathfrak{so}_{N+1})$ in section 4. We also show the relation between the random surface growth and the non-commutative random walk in section 4. In section 5, we show that moments of the random surface converge to Gaussian free fields along both space-like paths and time-like paths. In addition, the explicit covariance formulas are given.

\textbf{Acknowledgments.} The author would like to thank Jeffrey Kuan for enlightening discussions and
providing Figure \ref{fig:2}.

\section{Preliminaries}
\subsection{Representation theory}
We first review some useful results from representation theory (see e.g \cite{borodin2011random,MR2355506,OO-Limits}). 
In this paper, for any positive integer $N$ such that $N=2n+a-\frac{1}{2}$ with $a\in\left\{\pm \frac{1}{2}\right\}$,  we abuse the notation for $N$ and the corresponding pair $(n,a)$. 

Let $O(N+1)$ denote the group of $(N+1)\times (N+1)$ real-valued orthogonal matrices. The special orthogonal group $SO(N+1)$ is the subgroup of $O(N+1)$ consisting of matrices with determinant $1$. The associated Lie algebra is denoted by $\mathfrak{so}_{N+1}$, which consists of  $(N+1)\times (N+1)$  skew-symmetric square matrices with Lie bracket the commutator. Last, let $O(\infty)=\bigcup_{N=0}^\infty O(N+1)$ and $SO(\infty)=\bigcup_{N=0}^\infty SO(N+1)$.

Recall that if $O\in SO(N+1)$, then the spectrum of $O$ is of the form $\{z_1,z_1^{-1},\ldots,z_n,z_n^{-1}\}$ if $N+1=2n$, while when $N+1=2n+1$, the spectrum of $O$ is of the form $\{z_1,z_1^{-1},\ldots,z_{n},z_{n}^{-1},1\}$, where $z_i$ are roots of unity in both cases.
Let $\Omega$ be the set  of all $\omega=(\alpha,\beta,\delta)$ such that 
\begin{equation*}
    \alpha=(\alpha_1\ge \alpha_2\ge \ldots\ge 0)\in\mathbb{R}^{\infty},\     \beta=(\beta_1\ge \beta_2\ge \ldots\ge 0)\in\mathbb{R}^{\infty},\ \delta \in\mathbb{R}
\end{equation*}
and
\begin{equation*}
    \sum_{i=1}^\infty(\alpha_i+\beta_i)\le \delta.
\end{equation*}
For any $\omega=(\alpha,\beta,\delta)\in\Omega$, we define a function
\begin{equation}\label{eq: 5}
    \chi^\omega(O)=\prod_{j=1}^n E^\omega\left(\frac{z_j+z_j^{-1}}{2}\right),
\end{equation}
where
\begin{equation}\label{eq: 11}
    E^\omega(x)=e^{\left(\delta-\sum_{i=1}^\infty(\alpha_i+\beta_i)\right)(x-1)}\prod_{i=1}^\infty\frac{1-\beta_i(1-x)+\beta_i^2(1-x)/2}{1-\alpha_i(1-x)+\alpha_i^2(1-x)/2},
\end{equation}
if we set $x=(z+z^{-1})/2$. Then each $\omega\in\Omega$ identifies an extreme character $\chi^\omega$ of $O(\infty)$.

When $n\in\mathbb{N}$, it is a classical result that the set of all irreducible representations of $SO(2n+1)$ over $\mathbb{C}$ is parameterized by partitions of length $\le n$, which is a sequence of nonincreasing nonnegative integers $\lambda=(\lambda_1\ge \ldots\ge \lambda_{n}\ge 0)$. For each partition  $\lambda$, denote the corresponding character of the irreducible representation of $SO(2n+1)$ by $\chi^\lambda_{SO(2n+1)}$, whose dimension is denoted $\text{dim}_{SO(2n+1)}\lambda$. Similarly, the set of all irreducible representations of $SO(2n)$ over $\mathbb{C}$ is parameterized by sequences of integers $\lambda=(\lambda_1, \ldots, \lambda_n)$   satisfying $\lambda_1\ge \ldots\ge\lambda_{n-1}\ge |\lambda_n|$. Again,  for each $\lambda$, we denote the corresponding character of the irreducible representation of $SO(2n)$ by $\chi^\lambda_{SO(2n)}$ with dimension $\text{dim}_{SO(2n)}\lambda$.

Let $\mathbb{J}_n$ denotes the set of all partitions of length $\le n$ and  $\mathbb{J}_{n,a}$ with $a\in\{\pm 1/2\}$ be two copies of $\mathbb{J}_n$.  In what follows, we slightly abuse notations to interchange the use of  $a=\pm 1/2$ and its sign, for example, we will write $\mathbb{J}_{n,\pm 1/2}$ as $\mathbb{J}_{n,\pm}$.  

For any $\omega\in \Omega$, the restriction of $\chi^\omega$ defined in \eqref{eq: 5} to any $SO(M)$ defines two measures $P^\omega_{n,a}$ on $\mathbb{J}_{n,a}$ such that
\begin{equation}\label{eq: 7}
    \chi^\omega|_{SO(2n+1)}=\sum_{\lambda\in \mathbb{J}_n}P^\omega_{n,\frac{1}{2}}(\lambda)\frac{\chi^\lambda_{SO(2n+1)}}{\text{dim}_{SO(2n+1)}\lambda},
\end{equation}
\begin{equation}\label{eq: 8}
     \chi^\omega|_{SO(2n)}=\sum_{\lambda\in \mathbb{J}_n}P^\omega_{n,-\frac{1}{2}}(\lambda)\frac{\chi^\lambda_{SO(2n)}+\chi^{\lambda^*}_{SO(2n)}}{\text{dim}_{SO(2n)}\lambda}, 
\end{equation}
where $\lambda^*=(\lambda_1, \ldots, \lambda_{n-1},-\lambda_n)$.

Let $J_k^{(a,b)}$ denote the $k$-th Jacobi polynomial with parameters $a,b$ (see e.g. \cite{SG-Orthogonal}), define the constant $c_k$ to be 
\begin{equation*}
    c_k=\left\{\begin{array}{cc}
    \frac{1\cdot 3\cdot \ldots\cdot (2k-1)}{2\cdot 4\cdot\ldots\cdot 2k}     &  \text{ if } k>0,\\
        1 & \text{ if } k=0,
    \end{array}\right.
\end{equation*}
and let $\mathcal{J}_k^{(a,b)}=J_k^{(a,b)}/c_k$.

Then for any $\lambda\in\mathbb{J}_{n,a}$,  and $\omega\in\Omega$,  the measure $P^\omega_{n,a}$  provided in \cite{borodin2011random} is as follows:  
\begin{equation}\label{eq: 6}
   P^\omega_{n,a}(\lambda)=C_{n,a}\times \text{det}\left[\psi_{n-j}^{n,a}\left(\lambda_i-i+n\right)\right]_{1\le i,j\le n}\times \text{dim}_{SO(2n+1/2+a)}\lambda,
\end{equation}
where 
\begin{equation*}
    \psi_{n-l}^{n,a}(k)=\frac{W^{(a,-1/2)}(k)}{\pi}\int_{-1}^1E^\omega(x)(x-1)^{n-l}\mathcal{J}_k^{(a,-1/2)}(x)(1-x)^a(1+x)^{-1/2}dx,
\end{equation*}
\begin{equation*}
      W^{(a,b)}(k)=\left\{\begin{array}{ll}
    2, &  \text{ if } k>0,\ a=b=-1/2,\\
        1, & \text{ if } k=0,\ a=b=-1/2, \\
         1, & \text{ if } k\ge 0,\ a=1/2,\ b=-1/2,
    \end{array}\right.
\end{equation*}
and $C_{n,a}$ is the normalization term for $P^\omega_{n,a}$. In particular, $P^{(0,0,0)}_{n,a}$ is the delta measure at the point $(1,2,\cdots,n)$.

For $N+1=2n+1$ or $N+1=2n$ respectively, we number the rows and columns of $(N+1)\times (N+1)$ matrices by the indices $\{-n,\ldots,-1,0,1,\ldots ,n\}$ and  $\{-n,\ldots,-1,1,\ldots ,n\}$.

The Lie algebra $\mathfrak{so}_{N+1}$ is spanned by the generators
$$F_{ij}=E_{ij}-E_{-j,-i},\ \ \ -n\le i,j\le n, i\neq -j.$$
The commutation relation in $ \mathfrak{so}_{N+1}$ is 
\begin{equation*}
    [F_{ij},F_{kl}]=\delta_{kj}F_{il}-\delta_{il}F_{kj}-\delta_{i,-k}F_{-j,l}+\delta_{-l,j}F_{k,-i}.
\end{equation*}

The coproduct $\Delta:U(\mathfrak{so}_{N+1})\xrightarrow[]{}U(\mathfrak{so}_{N+1})\otimes U(\mathfrak{so}_{N+1})$
is given by $\Delta(F_{ij})=F_{ij}\otimes 1+1\otimes F_{ij}$.

Let $Z(U(\mathfrak{so}_{N+1}))$ be the centre of  $U(\mathfrak{so}_{N+1})$, we introduce the generators of $Z(U(\mathfrak{so}_{N+1}))$ (see e.g. chapter 7 of \cite{MR2355506}).
Denote by $F^{(m)}$ the matrix with entries $F_{ij}$, where $1\le m\le n$ and $i,j=-m,-m+1,\cdots,m$. Denote by $\Tilde{F}^{(m)}$ the submatrix of $F^{(m)}$ obtained by removing the row and column enumerated by $-m$. Let $\rho_i=-\rho_{-i}=-i+1$ if $N+1$ is even, while $\rho_i=-\rho_{-i}=-i+\frac{1}{2}$ if $N+1$ is odd.

Let $\Lambda_k^{(m)}$, $\Phi_k^{(m)}$,  $\Tilde{\Lambda}_k^{(m)}$, $\Tilde{\Phi}_k^{(m)}$  denote the noncommutative symmetric  functions defined as the coefficients in the expansion of the following quasideterminant (see Definition 1.10.1 in \cite{MR2355506}):
\begin{equation*}
    \begin{array}{l}
     1+\sum_{k=1}^\infty \Lambda_k^{(m)}q^k=\left|1+q\left(F^{(m)}+\rho_m\right)\right|_{mm},      \\
   \sum_{k=1}^\infty\Phi_k^{(m)}q^{k-1}=-\frac{d}{dq} log\left|1-q\left(F^{(m)}+\rho_m\right)\right|_{mm},       \\
       1+\sum_{k=1}^\infty \Tilde{\Lambda}_k^{(m)}q^k=\left|1+q\left(-\Tilde{F}^{(m)}-\rho_m\right)\right|_{mm},      \\
   \sum_{k=1}^\infty\Tilde{\Phi}_k^{(m)}q^{k-1}=-\frac{d}{dq} log\left|1-q\left(-\Tilde{F}^{(m)}-\rho_m\right)\right|_{mm}.        \\    
          
    \end{array}
\end{equation*}
In addition, there is a graphical construction of  $\Lambda_k^{(m)}$, $\Phi_k^{(m)}$,  $\Tilde{\Lambda}_k^{(m)}$, $\Tilde{\Phi}_k^{(m)}$.

For a $M\times M$ matrix $A$, we consider  the complete oriented graph $\mathcal{A}$ with vertices {1,2,\ldots,M} and label the arrow from vertex $i$ to $j$ by $a_{ij}$. For each directed path of length $k$ in graph $\mathcal{A}$ which starts from vertex $i$ and ends at vertex $j$, we  define a monomial of the form $A_{i,r_1}A_{r_1,r_2}\cdots A_{r_{k-1,j}}$.   

Let $A=F^{(m)}+\rho_m$, then $(-1)^{k-1}\Lambda_k^{(m)}$ is the sum of all monomials labeling simple path in $\mathcal{A}$ of length $k$ going from $m$ to $m$.  $\Phi_{2k}^{(m)}$ is the sum of all monomials labeling path in $\mathcal{A}$ of length $k$ going from $m$ to $m$, the coefficient of each monomial being the ratio of $k$ to the number of returns to $m$. Similarly, $\Tilde{\Lambda}_k^{(m)}$ and $\Tilde{\Phi}_k^{(m)}$ are defined with $A=-\Tilde{F}^{(m)}-\rho_m$.

Now, we define $\Lambda_{2n}^{N+1}=\prod_{m=1}^n \Tilde{\Lambda}_1^{(m)} \Lambda_1^{(m)}$ and $\Phi_{2k}^{N+1}=\sum_{m=1}^n\left(\Tilde{\Phi}_{2k}^{(m)}+\Phi_{2k}^{(m)}\right)$.
\begin{Example}
\begin{enumerate}[{(1)}]
    \item \begin{equation*}
    \Phi_2^{N+1}=2\sum_{m=1}^n\left(\left(F_{mm}+\rho_m\right)^2+2\sum_{-m<i<m}F_{mi}F_{im}\right).
\end{equation*}
\item\begin{multline*}
   \Phi_4^{N+1}=2\sum_{m=1}^n\Bigg(\left(F_{mm}+\rho_m\right)^4+2\sum_{-m<i,j<m}F_{mi}F_{im}F_{mj}F_{jm}   \\
   +2\sum_{-m<i,j<m}\left(F_{mm}+\rho_m\right)F_{mi}\left(F_{ij}+\delta_{ij}\rho_m\right)F_{jm}+2\sum_{-m<i,j<m}F_{mi}\left(F_{ij}+\delta_{ij}\rho_m\right)F_{jm}\left(F_{mm}+\rho_m\right)
\\
    +4\sum_{-m<i,j<m}F_{mi}\left(F_{ij}+\delta_{i,j}\rho_m\right)\left(F_{ji}+\delta_{i,j}\rho_m\right)F_{im}
    +2\sum_{-m<i<m}F_{mi}F_{i-m}F_{-mi}F_{im}\\
    +\frac{4}{3}\sum_{-m<i<m}\left(F_{mi}F_{im}\left(F_{mm}+\rho_m\right)^2+\left(F_{mm}+\rho_m\right)^2F_{mi}F_{im}+\left(F_{mm}+\rho_m\right)F_{mi}F_{im}\left(F_{mm}+\rho_m\right)\right)\Bigg).   
\end{multline*}
\end{enumerate}
\end{Example}
 
In Chapter 7 of \cite{MR2355506}, the explicit generators for $Z(U(\mathfrak{so}_{N+1}))$  were identified  using the Harish-Chandra isomorphism with the ring of shifted symmetric polynomials.

\begin{Theorem}\cite{MR2355506}
When $N=2n$, $Z\left(U\left(\mathfrak{so}_{N+1}\right)\right)$ is generated by $1$ and $\left\{\Phi_{2k}^{N+1}\right\}_{k=1}^n$. When $N=2n-1$, $Z(U(\mathfrak{so}_{N+1}))$ is generated by $1$,  $\left\{\Phi_{2k}^{N+1}\right\}_{k=1}^{n-1}$ together with $\sqrt{\Lambda_{2n}^{N+1}}$. Images of $\Phi_{2k}^{N+1}$ and $\Lambda_{2n}^{N+1}$ under the Harish-Chandra isomorphism are 
\begin{equation*}
     \begin{array}{lr}
     \frac{\Phi_{2k}^{N+1}}{2} \xrightarrow{} l_1^{2k}+\cdots+l_{n}^{2k},   & 
      \Lambda_{2n}^{N+1}\xrightarrow[]{} (-1)^nl_1^2\cdots l_n^2,    
      \end{array}
\end{equation*}
where $l_i=\lambda_i+\rho_i$.
\end{Theorem}
Thus, any $Y\in Z(U(\mathfrak{so}_{N+1}))$ acts as a multiplication by a scalar $p_Y(\lambda_1,\ldots,\lambda_n)$ in $V_\lambda$. When $N=2n$, $p_Y(\lambda_1,\ldots,\lambda_n)$ is symmetric in the variables $l_1^2,\ldots,l_n^2$. When $N=2n-1$,  $p_Y(\lambda_1,\ldots,\lambda_n)$ is the sum of a symmetric polynomial in $l_1^2,\ldots,l_n^2$ and $l_1\cdots l_n$ times a symmetric polynomial in $l_1^2,\ldots,l_n^2$.

\subsection{Random surface growth with a wall }

In this subsection, we provide more details about the model. 

Consider the two-dimensional lattice $\mathbb{Z}_{\ge 0}\times \mathbb{Z}_+$. On each horizontal level $\mathbb{Z}_+\times\{N\} $, there are $n=\left\lfloor\frac{N+1}{2}\right\rfloor$ particles. 

Denote the horizontal coordinates of all particles with vertical coordinate $N$ by $y_1^N>y_2^N>\cdots>y^N_n$. The reflecting wall forces $y^N_n\ge a+\frac{1}{2}$.   We have particle configurations
$$\{y_k^N\in \mathbb{Z}_{\ge 0}|k=1,2,\ldots,n; N=1,2,\ldots\}.  $$
Additionally, the particles satisfy the interlacing property $y_{k+1}^{N+1}<y_k^N<y_k^{N+1}$ for all meaningful values of $k$ and $N$.

The densely packed initial condition is given by the configuration where $y_k^N=N-2k+1$ for all $k$ and $N$, which means all the particles are as much to the left as possible. Each particle has two exponential  clocks of rate $\frac{1}{2}$, one is responsible for left jump and another for right jump. All clocks are independent. When the clock for particle $y_k^N$ rings, the particle tries to jump by $1$ in the corresponding direction. Right jumps are blocked if $y_k^N+1=y_{k-1}^{N-1}$ and left jumps are blocked if $y_k^N-1=y_k^{N-1}$. If the particle is against the wall (i.e. $y^N_n=0$) and the left jump clock rings, the particle is reflected and tries to jump to the right.

When $y_k^N$ tries to jump to the right and the jump is not blocked, we find the largest non-negative $r $ such that $y_k^{N+i}=y_k^N+i$ for $0\le i\le r$, and all particles $\{y_k^N\}_{i=0}^r$ jump to right  by $1$ simultaneously. If $y_k^N$ tries to jump to the left and is not blocked, we find the largest $l$ such that $y_{k+j}^{N+j}=y_k^N-j$ for $0\le j\le l$, and all particles $\{y_{k+j}^{N+j}\}_{j=0}^l$ jump to left by $1$.

In words, in order to maintain the interlacing property, the jump of one particle is blocked by particles below it, but if the jump is not blocked, the particle pushes particles above it to jump together. Additionally, jumps are reflected against the wall.

We then recall several equivalent ways to define the particle configurations of our model. 

It is  worth mentioning that the particle configurations of the above point process are obtained as an equivalent interpretation of paths of partitions \cite{borodin2011random}.

When $\lambda\in \mathbb{J}_{n,-}$ and $\mu\in \mathbb{J}_{n,+}$, we write $\lambda\prec\mu$ if $0\le \lambda_n\le \mu_n\le\ldots\le \lambda_1\le \mu_1$. When $\lambda\in \mathbb{J}_{n,+}$ and $\mu\in \mathbb{J}_{n+1,-}$, we write $\lambda\prec\mu$ if $0\le\mu_{n+1}\le \lambda_n\le \mu_n\le\ldots\le \lambda_1\le \mu_1$. Next, set $\mathbb{J}=\bigcup_{n\ge 1}\left(\mathbb{J}_{n,+}\cup\mathbb{J}_{n,-}\right)$.
 A path in $\mathbb{J}$ is a sequence $\mathbf{\lambda}=\left(\lambda^{(1),-}\prec\lambda^{(1),+}\prec\lambda^{(2),-}\prec\ldots\right)$ such that $\lambda^{(i),\pm}\in \mathbb{J}_{i,\pm}$. Last, let $\mathbb{J}^{N}$ to be the set of finite paths in $\mathbb{J}$ of length $N=2n-1/2+a$.

Set $\mathfrak{X}=\mathbb{Z}_{\ge 0}\times\mathbb{Z}_{>0}\times\{\pm \frac{1}{2}\}$ and $\mathfrak{Y}=\mathbb{Z}_{\ge 0}\times\mathbb{Z}_{>0}$ . $\mathfrak{Y}$ is equivalent to $\mathfrak{X}$ via the bijection sending $(x,n,a)\in \mathfrak{X}$ to $(2x+a+\frac{1}{2},2n+a-\frac{1}{2})\in\mathfrak{Y}$.

To any path $\mathbf{\lambda}=\left(\lambda^{(1),-}\prec\lambda^{(1),+}\prec\lambda^{(2),-}\prec\ldots\right)\in\mathbb{J}$, we associate  point configurations in $\mathfrak{X}$ and $\mathfrak{Y}$ as follows:
\begin{equation*}
    \mathcal{L}_{\mathfrak{X}}(\mathbf{\lambda})=\left\{\left(x_k^{n,a},n, a\right):\ 1\le k\le n,\ a\in\left\{\pm 1/2\right\},\ n\ge 1\right\},
\end{equation*}
\begin{equation*}
    \mathcal{L}_{\mathfrak{Y}}(\mathbf{\lambda})=\left\{\left(y_k^{2n-1/2+a},2n-1/2+a\right):\ 1\le k\le n,\ a\in\left\{\pm 1/2\right\},\ n\ge 1\right\},
\end{equation*}
where $x_k^{n,a}=\lambda_k^{n,a}+n-k$ and $y^{2n-1/2+a}_k=2(\lambda_k^{n,a}+n-k)+a+1/2$. In what follows, we interchange the use of $x_k^{n,a}$ and $y^{2n-1/2+a}_k$ for ease of notation.

We also adopt the notations and definitions from \cite{borodin2011random} as follows.

Define $T_{n,a}^{\varphi_t}$ on $\mathbb{J}_n\times\mathbb{J}_n$ by
\begin{equation*}
    T_{n,a}^{\varphi_t}(\mu,\lambda)=\text{det}\left[I_a^{\varphi_t}(\mu_i-i+n,\lambda_j-j+n)\right]_{1\le i,j\le n}\frac{\text{dim}_{SO(2n+1/2+a)}\lambda}{\text{dim}_{SO(2n+1/2+a)}\mu},
\end{equation*}
where $I_a^{\varphi_t}$ is given by
\begin{equation*}
    I_a^{\varphi_t}(l,i)=\frac{W^{\left(a,-\frac{1}{2}\right)}(i)}{\pi}\int_{-1}^1\mathcal{J}_i^{(a,-\frac{1}{2})}(x)\mathcal{J}_l^{(a,-\frac{1}{2})}(x)\varphi_t(x)(1-x)^a(1+x)^{-1/2}dx.
\end{equation*}
Define the matrix $T^{n,+}_{n,-}$ on $\mathbb{J}_{n,+}\times\mathbb{J}_{n,-}$ by
\begin{equation*}
  T^{n,+}_{n,-}(\mu,\lambda)=\text{det}\left[\phi^{+}_{-}(\mu_i-i+n,\lambda_j-j+n)\right]_{1\le i,j\le n}\frac{\text{dim}_{SO(2n)}\lambda}{\text{dim}_{SO(2n+1)}\mu},  
\end{equation*}
where
\begin{equation*}
   \phi^{+}_{-}(x,y)=\left\{\begin{array}{cc}
      1   & x\ge y=0,\\
       2  & x\ge y>0,\\
       0 & x<y.
   \end{array}\right. 
\end{equation*}
Define the matrix $T^{n,-}_{n-1,+}$ on $\mathbb{J}_{n,-}\times\mathbb{J}_{n-1,+}$ by
\begin{equation*}
  T^{n,-}_{n-1,+}(\mu,\lambda)=\text{det}\left[\phi^{-}_{+}(\mu_i-i+n,\lambda_j-j+n-1)\right]_{1\le i,j\le n}\frac{\text{dim}_{SO(2n-1)}\lambda}{\text{dim}_{SO(2n)}\mu},  
\end{equation*}
where
\begin{equation*}
    \phi^{-}_{+}(x,y)=\left\{\begin{array}{cc}
      1   & x> y,\\
      0   & x\le y.
   \end{array}\right. 
\end{equation*}

For each measure $P^\omega_{n,a}$ in \eqref{eq: 6}, we define the measure
\begin{equation}\label{eq: 2}
P^{N,\omega}=P^\omega_{n,a}\left(\lambda^{(n),a}\right)T^{n,a}_{n-1/2+a,-a}\left(\lambda^{(n),a},\lambda^{(n-1/2+a),-a}\right)\cdots T^{2,-}_{1,+}\left(\lambda^{(2),-1/2},\lambda^{(1),1/2}\right)T^{1,+}_{1,-}\left(\lambda^{(1),1/2},\lambda^{(1),-1/2}\right)
\end{equation}
on $\mathbb{J}^{N}$. When $\omega=(0,0,0)$,  $P^{N,\omega}$ is the delta measure at the densely packed initial condition, i.e. $x^{m,a}_k=m-k$ for all $1\le k\le m$ and $1\le m\le n$.

Define a stochastic matrix $A^{\varphi_t}_{N}$ on $\mathbb{J}^{N}\times\mathbb{J}^{N} $ by  
\begin{multline*}
   A_{N}^{\varphi_t}(\mathbf{\lambda},\mathbf{\mu})= \frac{T^{\varphi_t}_{n,a}\left(\lambda^{(n),a},\mu^{(n),a})T^{n,a}_{n-1/2+a,-a}(\mu^{(n),a},\mu^{(n-1/2+a),-a}\right)}{T^{n,a}_{(n-1/2+a),-a}T^{\varphi_t}_{(n-1/2+a),-a}\left(\lambda^{(n),a},\mu^{(n-1/2+a),-a}\right)}\times\cdots      \\
\cdots\frac{T^{\varphi_t}_{1,+}\left(\lambda^{(1),1/2},\mu^{(1),1/2}\right)T^{1,+}_{1,-}\left(\mu^{(1),1/2},\mu^{(1),-1/2}\right)}{T^{1,+}_{1,-}T^{\varphi_t}_{1,-}\left(\lambda^{(1),1/2},\mu^{(1),-1/2}\right)}  T^{\varphi_t}_{1,-}\left(\lambda^{(1),-1/2},\mu^{(1),-1/2}\right).
\end{multline*}

If we let $Q_{N}$ be the generator of the random surface growth with a wall  restricted to $\mathbb{J}^N$, then the continuous-time Markov chain can be characterized by matrix $A^{\varphi_t}_{N}$.
\begin{Theorem}(Theroem 3.12 in \cite{borodin2011random})
Let $\varphi_t(x)=e^{t(x-1)}$, then
\begin{equation*}
    e^{tQ_{N}}P^{N,\omega}=A^{\varphi_t}_{N} P^{N,\omega}.
\end{equation*}
\end{Theorem}

\section{Correlations along space-like paths}
In this section, we compute the multi-time correlation functions for the Markov process on $\mathbb{J}^{N}$.

\begin{Proposition}\label{prop: 1}
 Consider the evolution of the measure $P^{N,\omega}$
on $\mathbb{J}^{N}$ defined in \eqref{eq: 2} under the Markov chain $A^{\varphi_t}_{N}$, and denote by $\left(x^{1,-1/2}(t),\ldots,x^{n,a}(t)\right)$ the result after time $t\ge 0$. Then for any 
\begin{equation*}
    0=t^{N}_0\le \cdots \le t^{N}_{c(N)}= t^{N-1}_{0}\le \cdots\le t^{N-1}_{c(N-1)}=\\
    t^{N-2}_{0}\le \cdots\le t^2_{c(2)}=t^1_0\le \cdots\le t^1_{c(1)},
\end{equation*}
where $c(i)$ are arbitrary non-negative integers, the joint distribution of
\begin{multline*}
      x^{n,a}\left(t^{N}_0\right),\ldots x^{n,a}\left(t^{N}_{c(N)}\right), x^{n-1/2+a,-a}\left(t^{N-1}_0\right),\ldots\\ 
       \ldots, x^{n-1/2+a,-a}\left(t^{N-1}_{c(N-1)}\right),x^{n-3/2-a,a}\left(t^{N-2}_{0}\right),\ldots, x^{1,1/2}\left(t^2_{c(2)}\right),x^{1,-1/2}\left(t^1_0\right),\ldots,x^{1,-1/2}\left(t^1_{c(1)}\right) 
\end{multline*} 
coincides with the stochastic evolution of $P^{N,\omega}$ under transition matrices
\begin{multline*}
        T^{\varphi_{t^{N}_1-t^{N}_0}}_{n,a},\ldots,T^{\varphi_{t^{N}_{c(N)}-t^{N}_{c(N)-1}}}_{n,a},T^{n,a}_{n-1/2+a,-a}, T^{\varphi_{t^{N-1}_1-t^{N-1}_0}}_{n-1/2+a,-a},\ldots,\\
        \ldots T^{\varphi_{t^{N-1}_{c(N-1)}-t^{N-1}_{c(N-1)-1}}}_{n-1/2+a,-a},T^{n-1/2+a,-a}_{n-3/2-a,a}, T^{\varphi_{t^{N-2}_1-t^{N-2}_0}}_{n-3/2-a,a},\ldots,T^{\varphi_{t^{1}_{c(1)}-t^{1}_{c(1)-1}}}_{1,+}.
\end{multline*}

\end{Proposition}
\begin{proof}
See Proposition 2.5 of \cite{2013} and Proposition 3.6 of \cite{borodin2011random}.
\end{proof}

\begin{Definition}
For any $M\ge 1$, pick $M$ points
\begin{equation*}
    \varkappa_j=(n_j,a_j,t_j,s_j)\in \mathbb{Z}_{> 0}\times\{\pm \frac{1}{2}\}\times \mathbb{R}_{\ge 0} \times\mathbb{Z}_{\ge 0} ,
\end{equation*}
Given the Markov process $x(t)$, The $M$-th correlation function $\rho_M$ at $(\varkappa_1,\ldots,\varkappa_M)$ is defined as 
\begin{multline*}
    \rho(\varkappa_1,\ldots,\varkappa_M)=\text{Prob}\Big(\text{For each }1\le j\le M, \text{ there exists a } k_j,
    1\le k_j\le n_j, \text{ such that } x_{k_j}^{n_j,a_j}(t_j)=s_j\Big) . 
\end{multline*}
  
\end{Definition}

Next we introduce  a partial order on  $(n,a,t)$ via $(n_1,a_1,t_1)\prec (n_2,a_2,t_2)$ if  $2n_1-1/2+a_1\le 2n_2-1/2+a_2$, $t_1\ge t_2$ and $(n_1,a_1,t_1)\neq (n_2,a_2,t_2)$.

\begin{Theorem}\label{thm: 1}
Consider the Markov process with the initial distribution $P^{N,\omega}$ given by \eqref{eq: 2},  let $\varphi_t(x)=e^{t(x-1)}$. Assume any two distinct pairs in  $\left\{\varkappa_j=(n_j,a_j,t_j,s_j)|1\le j\le M\right\}$  are comparable with respect to $\prec$. Then
\begin{equation}
     \rho(\varkappa_1,\cdots,\varkappa_M)=\text{det}\left[K(\varkappa_i,\varkappa_j)\right]_{i,j=1}^M,
\end{equation}
where $ K(\varkappa_i,\varkappa_j)$ is the correlation kernel given by


\begin{multline*}
   K(\varkappa_i,\varkappa_j)=  \\ \frac{W^{\left(a_1,-\frac{1}{2}\right)}(s_1)}{\pi}\frac{1}{2\pi i}\int_{-1}^1\oint_C\mathcal{J}_{s_1}^{\left(a_1,-\frac{1}{2}\right)}(y){J}_{s_2}^{\left(a_2,-\frac{1}{2}\right)}(u)
 \frac{E^\omega(y)}{E^\omega(u)} \frac{e^{t_1(y-1)}}{e^{t_2(u-1)}}
  \frac{(y-1)^{n_1}}{(u-1)^{n_2}}
  \frac{(1-y)^{a_1}(1+y)^{-1/2}}{y-u} dudy\\
    +\mathbb{1}_{(n_1,a_1,t_1)\succ (n_2,a_2,t_2)}\frac{W^{\left(a_1,-\frac{1}{2}\right)}(s_1)}{\pi}\int_{-1}^1\mathcal{J}_{s_1}^{\left(a_1,-\frac{1}{2}\right)}(y){J}_{s_2}^{\left(a_2,-\frac{1}{2}\right)}(y)
  \frac{e^{t_1(y-1)}}{e^{t_2(y-1)}}
  (y-1)^{n_1-n_2}
  (1-y)^{a_1}(1+y)^{-1/2}  dy .
\end{multline*}
The $u$-contour $C$ is a positively oriented simple loop that encircles the interval
$[-1, 1]$ but does not encircle any zeroes of $E^\omega$.
\end{Theorem}

\begin{Remark}
The correlation kernel in Theorem \ref{thm: 1} degenerates to the kernel in Theorem 4.1 of \cite{borodin2011random} when $t_1=t_2$, but with slightly different notations: In \cite{borodin2011random}, the time  was implicitly written in the parameter $\omega$, where they introduced a parameter $\gamma=\delta-\sum_{i=1}^\infty(\alpha_i+\beta_i)$ in the definition of function $E^\omega$ and let it to be the time parameter. In this paper, the function $E^\omega$ defined in \ref{eq: 11} is fixed as time changes. 
\end{Remark}

\subsection{Determinantal structure of the correlation functions}
The initial conditions for the Markov process are the distributions $P^{N,\omega}$. Proposition \ref{prop: 1} implies that the joint distribution of
\begin{multline*}
      x^{n,a}\left(t^{N}_0\right),\ldots x^{n,a}\left(t^{N}_{c(N)}\right), x^{n-1/2+a,-a}\left(t^{N-1}_0\right),\ldots\\ 
       \ldots, x^{n-1/2+a,-a}\left(t^{N-1}_{c(N-1)}\right),x^{n-3/2-a,a}\left(t^{N-2}_{0}\right),\ldots, x^{1,1/2}\left(t^2_{c(2)}\right),x^{1,-1/2}\left(t^1_0\right),\ldots,x^{1,-1/2}\left(t^1_{c(1)}\right) 
\end{multline*}   takes the form 
\begin{multline}\label{eq: 3}
   \text{const}\times \prod_{m=1}^{n}\Bigg[\text{det}\left[\phi_{+}^{-}\left(x_l^{m,-}\left(t^{2m-1}_{c(2m-1)}\right),x_k^{m-1,+}\left(t^{2m-2}_{0}\right)\right) \right]_{1\le k,l\le m}\\
   \times \prod_{b=1}^{c(2m-1)}\text{det}\left[\mathcal{T}_{t^{2m-1}_{b-1},t^{2m-1}_{b}}^{m,-}\left(x_l^{m,-}\left(t^{2m-1}_{b-1}\right),x_k^{m,-}\left(t^{2m-1}_{b}\right) \right)\right]_{1\le k,l\le m}\\
   \times \text{det}\left[\phi^{+}_{-}\left(x_l^{m,+}\left(t^{2m}_{c(2m)}\right),x_k^{m,-}\left(t^{2m-1}_{0}\right)\right) \right]_{1\le k,l\le m}\\
   \times \prod_{b=1}^{c(2m)}\text{det}\left[\mathcal{T}_{t^{2m}_{b-1},t^{2m}_{b}}^{m,+}\left(x_l^{m,+}\left(t^{2m}_{b-1}\right),x_k^{m,+}\left(t^{2m}_{b}\right)\right) \right]_{1\le k,l\le m}\Bigg]\\
   \times\text{det}\left[\psi_{n-l}^{n,a}\left(x_k^{n, a}\left(t^{N}_{0}\right) \right)\right]_{1\le k,l\le n},
\end{multline}
where $\mathcal{T}_{t_1,t_2}^{n,a}=I_a^{\varphi_{t_2-t_1}}$ and $x^{m-1,+}_{m}$ is set to be $-1$, we refer to this variable as $\text{virt}$.
If $a=-1/2$, then the last determinant of $\phi^{+}_{-}\left(x_l^{n,+}\left(t^{2n}_{c(2n)}\right),x_k^{n,-}\left(t^{2n-1}_{0}\right)\right)$ and the product of determinants of  $\mathcal{T}^{n,+}_{t^{2n}_{b-1},t^{2n}_b} $ do not occur.

With the above formula for the joint distribution, we apply an analogue of  Theorem 4.2 of \cite{borodin2008large} to compute the correlation kernel. Let $*$ denote convolution:
\begin{equation*}
    f*g(x,y)=\sum_{z\ge 0}f(x,z)g(z,y).
\end{equation*}

For any $n$, $a$ and any two time moments $ t_i^{2n-1/2+a}< t_j^{2n-1/2+a}$, define
\begin{equation*}
    \mathcal{T}^{n,a}_{t^{2n-1/2+a}_i,t_j^{2n-1/2+a}}=  \mathcal{T}^{n,a}_{t^{2n-1/2+a}_{i},t_{i+1}^{2n-1/2+a}}\ast\mathcal{T}^{n,a}_{t^{2n-1/2+a}_{i+1},t_{i+2}^{2n-1/2+a}}\ast\cdots\ast\mathcal{T}^{n,a}_{t^{2n-1/2+a}_{j-1},t_j^{2n-1/2+a}}
\end{equation*}
and
\begin{equation*}
     \mathcal{T}^{n,a}=\mathcal{T}^{n,a}_{t_{0}^{2n-1/2+a},t^{2n-1/2+a}_{c({2n-1/2+a})}}.
\end{equation*}

For any time moments $t^{2n_2-1/2+a_2}_{b_2}<t^{2n_1-1/2+a_1}_{b_1}$ with $n_1\le n_2$, we denote the convolution over all transitions between them by $\phi^{t^{2n_2-1/2+a_2}_{b_2},t^{2n_1-1/2+a_1}_{b_1}}$:


\begin{multline*}
 \phi^{t^{2n_2-1/2+a_2}_{b_2},t^{2n_1-1/2+a_1}_{b_1}}=  \mathcal{T}^{n_2,a_2}_{t^{2n_2-1/2+a_2}_{b_2},t^{2n_2-1/2+a_2}_{c(2n_2-1/2+a_2)}}\ast \phi_{-a_2}^{a_2} \ast   \mathcal{T}^{n_2-1/2+a_2,-a_2}\ast   \cdots\\
   \cdots\ast \phi_{a_1}^{-a_1} \ast   \mathcal{T}^{n_1,a_1}_{t^{2n_1-1/2+a_1}_{0},t^{2n_1-1/2+a_1}_{b_1}}. 
\end{multline*}
For example, when $a_1=-a_2=1/2$,
\begin{equation*}
   \phi^{t^{2n_2-1/2+a_2}_{b_2},t^{2n_1-1/2+a_1}_{b_1}}=  \mathcal{T}^{n_2,-}_{t^{2n_2-1}_{b_2},t^{2n_2-1}_{c(2n_2-1)}}\ast  \phi_{+}^{-} \ast   \mathcal{T}^{n_2-1,+}\ast \cdots\ast \phi_{+}^{-} \ast   \mathcal{T}^{n_1,+}_{t^{2n_1}_{0},t^{2n_1}_{b_1}}.      
\end{equation*}
When $t^{2n_1-1/2+a_1}_{b_1}\le t^{2n_2-1/2+a_2}_{b_2}$, let
\begin{equation*}
    \phi^{t^{2n_2-1/2+a_2}_{b_2},t^{2n_1-1/2+a_1}_{b_1}}=0.
\end{equation*}

For $a=\pm \frac{1}{2}$, we define matrices $M^{n,a}=[M^{n,a}]_{k,l=1}^{n}$ with entries 
\begin{equation*}
   M_{k,l}^{n,a}= \psi^{n,a}_{n-l} \ast \mathcal{T}^{n,a} \ast\phi_{-a}^{a}\ast    \mathcal{T}^{n-1/2+a,-a}\ast \cdots\ast \mathcal{T}^{k,-}\ast\phi_{+}^{-}(\text{virt}).
\end{equation*}
where $\text{virt}$ is the virtual variable for $\phi_{+}^{-}$, which is equal to $-1$.

Last, for $N=2n-1/2+a$ and $l\le n$, define  
\begin{equation*}
    \psi_{k-l}^{ t^{2k-1/2+a_1}_b,   N}=\psi^{n,a}_{n-l}\ast \phi^{t^{N}_{0},t^{2k-1/2+a_1}_{b}}.
\end{equation*}
\begin{Lemma}\label{Lemma: 1}
If $M^{n,a}$ is upper triangular and invertible,  there exist functions $\Phi_{k-j}^{ t^{2k-1/2+a_1}_b,N}(s)$ such that 
\begin{enumerate}[{(1)}]
    \item $\left\{\Phi_{k-j}^{ t^{2k-1/2+a_1}_b,N}(s)\right\}_{j=1}^k$ is a  basis of the linear span of 
    \begin{equation*}
        \left\{\mathcal{T}_{t^{2k-1/2+a_1}_b,t^{2k-1/2+a_1}_{c\left(2k-1/2+a_1\right)}}^{k,a_1}\ast\cdots\ast\mathcal{T}_{t^{2j-1}_0,t^{2j-1}_{c\left(2j-1\right)}}^{j,-}\ast \phi_{+}^{-}(s,\text{virt})\right\}_{j=1}^k.
    \end{equation*}
    \item For $0\le j_1,j_2 \le k-1$, 
    \begin{equation*}
        \sum_{s\ge 0} \Phi_{j_1}^{t^{2k-1/2+a_1}_b,N}(s)\psi_{j_2}^{t^{2k-1/2+a_1}_b,N}(s)=\delta_{j_1,j_2}.
    \end{equation*}
\end{enumerate}
Then the correlation kernel for two comparable pairs $(n_1,a_1,t_1)$ and $ (n_2,a_2,t_2)$ with $2n_i-1/2+a_i\le N$ is given by 
\begin{equation*}
    K(n_1,a_1,t_1,s_1;n_2,a_2,t_2,s_2)=-\phi^{t^{2n_2-1/2+a_2}_{b_2},t^{2n_1-1/2+a_1}_{b_1}}(s_1,s_2)+\sum_{k=1}^{n_2}\psi_{n_1-k}^{t^{2n_1-1/2+a_1}_{b_1},N}(s_1)\Phi_{n_2-k}^{ t^{2n_2-1/2+a_2}_{b_2},N}(s_2),
\end{equation*}
where $t^{2n_i-1/2+a_i}_{b_i}=t_i$, $i=1,2$. 
\end{Lemma}

\begin{proof}
The proof is similar to the proof of Theorem 4.2 in \cite{borodin2008large} and Lemma 3.4 of \cite{Borodin_2007}.

First we define a matrix $L^{n,a}$ such that the measure in \eqref{eq: 3} is proportional to a suitable symmetric minor of $L^{n,a}$.
\begin{equation*}
    L^{n,a}=\left(\begin{array}{cccccccccccc}
       0  &  E_1 & 0       & 0    & 0  & E_2  &\cdots &E_{n} & 0 & 0 & 0 \\
       0  &  0   &-\mathbb{T}^{1,-} &0     & 0   & 0  &\cdots& 0 & 0 & 0 & 0\\
        0  &  0   &0 &-F_{1,-}^{1,+}     & 0   & 0  &\cdots&  0 & 0 & 0 & 0\\
         0  &  0    &0     & 0 &-\mathbb{T}^{1,+}  & 0  &\cdots&  0 & 0 & 0 & 0\\
          0  &  0   &0 &0     & 0   & -F_{1,+}^{2,-}  &\cdots&  0 & 0 & 0 & 0\\
        \vdots & \vdots &      \vdots & \vdots &         \vdots & \vdots &\ddots    & \vdots & \vdots &     \vdots & \vdots \\  
          0  &  0   &0 &0     & 0   & 0 &\cdots& -F_{n-1,+}^{n,-} & 0 & 0 & 0\\
            0  &  0    &0     & 0   & 0 & 0 &\cdots & 0&-\mathbb{T}^{n,-} & 0 & 0\\
        0  &  0   &0    & 0   & 0 & 0 &\cdots & 0 & 0& -F_{n,-}^{n,+}  & 0\\
         0  &  0    &0     & 0   & 0 & 0 & \cdots & 0 & 0 & 0&-\mathbb{T}^{n,+}\\
       \Psi^{n,a} &  &  0   &0 &0     & 0   &\cdots  &0&  0 & 0 & 0 
    \end{array}\right).
\end{equation*}
If $a=-1/2$, the last two columns do not occur.
The matrix blocks in $L^{n,a}$ have the following entries:
\begin{equation*}
 \left[\Psi^{n,a}\right]_{k,l}=\psi_{n-l}^{n,a}\left(x_k^{n,a}\left(t^{N}_{0}\right)\right),\quad {1\le k,l\le n}.    
\end{equation*}
\begin{equation*}
 \left[E_{m}\right]_{kl}=\left\{\begin{array}{cc}
   \phi_{+}^{-}\left(x_l^{m,-}\left(t^{2m-1}_{c(2m-1)}\right),x_m^{m-1,+}\left(t^{2m-2}_{0}\right)\right),      & k=m,\ 1\le l\le m, \\
     0,    & 1\le k\le n,\ k\neq m,\ 1\le l\le m,
    \end{array}\right.
    \end{equation*}

\begin{equation*}
   \left[ F_{m,-}^{m,+}\right]_{k,l}=\phi^{+}_{-}\left(x_l^{m,+}\left(t^{2m}_{c(2m)}\right),x_k^{m,-}\left(t^{2m-1}_{0}\right)\right),\quad  {1\le k,l\le m},
\end{equation*}
\begin{equation*}
   \left[ F_{m-1,+}^{m,-}\right]_{k,l}=\phi_{+}^{-}\left(x_l^{m,-}\left(t^{2m-1}_{c(2m-1)}\right),x_k^{m-1,+}\left(t^{2m-2}_{0}\right)\right),\quad {1\le k\le m-1,\ 1\le l\le m},
\end{equation*}
and $\mathbb{T}^{m,\pm}$ is the matrix made of blocks
\begin{equation*}
    \mathbb{T}^{m,\pm}=\left(\begin{array}{cccc}
     \mathbb{T}^{m,\pm}_1    & 0 & 0 \\
        0 & \ddots &0\\
        0&0& \mathbb{T}^{m,\pm}_{c(2m-1/2\pm1/2)}
    \end{array}\right),
\end{equation*}
where  
\begin{multline*}
  \left[\mathbb{T}^{m,\pm}_b\right]_{k,l}=\mathcal{T}_{t^{2m-1/2\pm 1/2}_{b-1},t^{2m-1/2\pm 1/2}_b}^{m,\pm} \left(x_l^{m,\pm} \left(t^{2m-1/2\pm1/2}_{b-1}\right),x_k^{m,\pm} \left(t^{2m-1/2\pm1/2}_{b}\right)\right),\\
  1\le k,l\le m,\ 1\le b\le c(2m-1/2\pm 1/2).   
\end{multline*}

The rest of the proof is along the same lines as that of Lemma 3.4 in \cite{Borodin_2007}.
\end{proof}

In the following sections, we compute each term that appears in Lemma \ref{Lemma: 1}.  We first calculate the function $  \phi^{t^{2n_2-1/2+a_2}_{b_2},t^{2n_1-1/2+a_1}_{b_1}}$ in section 3.2, then we calculate  matrix $M^{n,a}$ and the function $\psi_{n-k}^{ t^{2k-1/2+a_1}_b,N}$ in section 3.3. In section 3.4, we compute $ \Phi_{k-j}^{ t^{2k-1/2+a_1}_b,N}(s)$ and $\displaystyle {\sum_{k=1}^{n_2}\psi_{n_1-k}^{t^{2n_1-1/2+a_1}_{b_1},N}(s_1)\Phi_{n_2-k}^{ t^{2n_2-1/2+a_2}_{b_2},N}(s_2)}$. Last, in section 3.5, we combine all the functions obtained from section 3.3--3.4 to  get the correlation kernel $K$ in Theorem \ref{thm: 1}.

\subsection{Calculating $  \phi^{t^{2n_2-1/2+a_2}_{b_2},t^{2n_1-1/2+a_1}_{b_1}}$}
Recall that $\varphi_{t}(x)=e^{t(x-1)}$,  for ease of notation, let $\varphi_{n,+}(x)=\varphi_{t^{2n}_{c(2n)}-t^{2n}_{0}}(x)$, $\varphi_{n,-}(x)=\varphi_{t^{2n-1}_{c(2n-1)}-t^{2n-1}_{0}}(x)$ and $\varphi_{n,\pm}(x)=\varphi_{n,+}(x)\varphi_{n,-}(x)=\varphi_{t^{2n}_{c(2n)}-t^{2n-1}_{0}}(x)$.
Let's recall some useful identities from Lemma 2.2--2.7 in \cite{borodin2011random}.
\begin{Lemma}\label{Lemma: 2}\cite{borodin2011random} Let $T(x)\in C^1[-1,1]$, the following identities hold:
\begin{enumerate}[{(1)}]
    \item \begin{equation*}
        \sum_{r=0}^s W^{\left(-\frac{1}{2},-\frac{1}{2}\right)}(r)\mathcal{J}_r^{\left(-\frac{1}{2},-\frac{1}{2}\right)}(x)=\mathcal{J}_s^{\left(\frac{1}{2},-\frac{1}{2}\right)}(x),
    \end{equation*}
    \item\begin{equation*}
        \frac{1}{\pi}\int_{-1}^1 \mathcal{J}_{r}^{\left(\frac{1}{2},-\frac{1}{2}\right)}(x)(1-x)^{-1/2} (1+x)^{-1/2}  dx=1,
    \end{equation*}
    \item 
    \begin{equation*}
        \sum_{r=0}^{\infty} \frac{W^{\left(-\frac{1}{2},-\frac{1}{2}\right)}(r)}{\pi}\int_{-1}^1 \mathcal{J}_{r}^{\left(-\frac{1}{2},-\frac{1}{2}\right)}(x)T(x)(1-x)^{-1/2} (1+x)^{-1/2}  dx=T(1),
    \end{equation*}
    \item
    \begin{multline*}
          \sum_{r=s+1}^{\infty} \frac{W^{\left(-\frac{1}{2},-\frac{1}{2}\right)}(r)}{\pi}\int_{-1}^1 \mathcal{J}_{r}^{\left(-\frac{1}{2},-\frac{1}{2}\right)}(x)T(x)(1-x)^{-1/2} (1+x)^{-1/2}  dx\\
          =\frac{1}{\pi}\int_{-1}^1 \mathcal{J}_{s}^{\left(\frac{1}{2},-\frac{1}{2}\right)}(x)\left(T(1)-T(x)\right)(1-x)^{-1/2} (1+x)^{-1/2}  dx,
    \end{multline*}
    \item
    \begin{multline*}
          \sum_{r=s}^{\infty} \frac{W^{\left(\frac{1}{2},-\frac{1}{2}\right)}(r)}{\pi}\int_{-1}^1 \mathcal{J}_{r}^{\left(\frac{1}{2},-\frac{1}{2}\right)}(x)T(x)(1-x)^{1/2} (1+x)^{-1/2}  dx\\
          =\frac{1}{\pi}\int_{-1}^1 \mathcal{J}_{s}^{\left(-\frac{1}{2},-\frac{1}{2}\right)}(x)T(x)(1-x)^{-1/2} (1+x)^{-1/2}  dx,
    \end{multline*}
     \item for $a=\pm \frac{1}{2}$, $-1\le \zeta\le 1$,
    \begin{equation*}
        \sum_{k=0}^{\infty} \frac{W^{\left(a,-\frac{1}{2}\right)}(k)}{\pi}\int_{-1}^1 \mathcal{J}_{k}^{\left(a,-\frac{1}{2}\right)}(x)\mathcal{J}_{k}^{\left(a,-\frac{1}{2}\right)}(\zeta)T(x)(1-x)^{a} (1+x)^{-1/2}  dx=T(\zeta).
    \end{equation*}
\end{enumerate}
\end{Lemma}

We start by computing some basic convolutions that will be useful later.
\begin{Proposition}\label{Prop: 1}

\begin{enumerate}[{(1)}]
    \item \begin{equation*}
    \mathcal{T}^{n,a}_{t^{2n-1/2+a}_i,t_j^{2n-1/2+a}} =I_a^{\varphi_{t^{2n-1/2+a}_j-t_i^{2n-1/2+a}}}.  
\end{equation*}
\item\begin{equation*}
  \mathcal{T}^{n,+} \ast\phi_{-}^{+} (m,l)= \frac{W^{\left(-\frac{1}{2},-\frac{1}{2}\right)}(l)}{\pi}\int_{-1}^1 \mathcal{J}_{l}^{\left(-\frac{1}{2},-\frac{1}{2}\right)}(x)\mathcal{J}_{m}^{\left(\frac{1}{2},-\frac{1}{2}\right)}(x)\varphi_{n,+}(x)(1-x)^{-1/2} (1+x)^{-1/2}  dx.
\end{equation*}
\item\begin{multline*}
    \mathcal{T}^{n,-} \ast\phi_{+}^{-}(m,l)=\\
  \left\{\begin{array}{cc} 
   \begin{split}
        \\ \frac{W^{\left(\frac{1}{2},-\frac{1}{2}\right)}(l)}{\pi}\int_{-1}^1 \mathcal{J}_{l}^{\left(\frac{1}{2},-\frac{1}{2}\right)}(x)\left(\mathcal{J}_{m}^{\left(-\frac{1}{2},-\frac{1}{2}\right)}(1)\varphi_{n,-}(1)-\mathcal{J}_{m}^{\left(-\frac{1}{2},-\frac{1}{2}\right)}(x)\varphi_{n,-}(x)\right)\times\\
    (1-x)^{-1/2} (1+x)^{-1/2}  dx,      \end{split}    &  m,l \ge 0,\\
   \mathcal{J}_{m}^{\left(-\frac{1}{2},-\frac{1}{2}\right)}(1)\varphi_{n,-}(1) ,   &  m\ge 0, \ l=\text{virt}.
  \end{array}\right.
\end{multline*}
\item\begin{multline*}
     \mathcal{T}^{n,+} \ast\phi_{-}^{+}\ast \mathcal{T}^{n,-} \ast\phi_{+}^{-}(m,l)=\\
  \left\{    \begin{array}{lc}
  \begin{split}
    -
     \frac{1}{\pi}\int_{-1}^1\mathcal{J}_{m}^{\left(\frac{1}{2},-\frac{1}{2}\right)}(y)\mathcal{J}_{l}^{\left(\frac{1}{2},-\frac{1}{2}\right)}(y)\varphi_{n,\pm}(y)(1-y)^{-1/2}(1+y)^{-1/2}  dy\\  +\mathcal{J}_{m}^{\left(\frac{1}{2},-\frac{1}{2}\right)}(1)\varphi_{n,\pm}(1),
  \end{split} &  m\ge 0,\ l\ge 0 ,  \\
  \mathcal{J}_{m}^{\left(\frac{1}{2},-\frac{1}{2}\right)}(1)\varphi_{n,\pm}(1), &  m\ge 0,\ l=\text{virt}.\\
    \end{array}\right.   
\end{multline*}
\end{enumerate}

\end{Proposition}

\begin{proof}
Results (1)--(3) follow directly from Lemma \ref{Lemma: 2}. We only show the computation for  (4). By definition, when $m,l\ge 0$,
\begin{multline*}
  \mathcal{T}^{n,+} \ast\phi_{-}^{+}\ast \mathcal{T}^{n,-} \ast\phi_{+}^{-}(m,l)=  \\ \frac{1}{\pi^2}\sum_{k=0}^{\infty}\left(W^{\left(-\frac{1}{2},-\frac{1}{2}\right)}(k) \int_{-1}^1 \mathcal{J}_{k}^{\left(-\frac{1}{2},-\frac{1}{2}\right)}(x)\mathcal{J}_{m}^{\left(\frac{1}{2},-\frac{1}{2}\right)}(x)\varphi_{n,+}(x)(1-x)^{-1/2} (1+x)^{-1/2}  dx\times\right.\\
 \left.\int_{-1}^1 \mathcal{J}_{l}^{\left(\frac{1}{2},-\frac{1}{2}\right)}(y)\left(\mathcal{J}_{k}^{\left(-\frac{1}{2},-\frac{1}{2}\right)}(1)\varphi_{n,-}(1)-\mathcal{J}_{k}^{\left(-\frac{1}{2},-\frac{1}{2}\right)}(y)\varphi_{n,-}(y)\right)(1-y)^{-1/2} (1+x)^{-1/2}  dy\right).   
\end{multline*}
According to Lemma \ref{Lemma: 2} (2),
\begin{equation*}
  \frac{1}{\pi}\int_{-1}^1 \mathcal{J}_{l}^{\left(\frac{1}{2},-\frac{1}{2}\right)}(y)\mathcal{J}_{k}^{\left(-\frac{1}{2},-\frac{1}{2}\right)}(1)\varphi_{n,-}(1)(1-y)^{-1/2} (1+y)^{-1/2}  dy =\mathcal{J}_{k}^{\left(-\frac{1}{2},-\frac{1}{2}\right)}(1)\varphi_{n,-}(1), 
\end{equation*}
thus we have

\begin{multline}\label{eq: 1}
     \mathcal{T}^{n,+} \ast\phi_{-}^{+}\ast \mathcal{T}^{n,-} \ast\phi_{+}^{-}(m,l)= \\ -\frac{1}{\pi^2}\sum_{k=0}^{\infty}W^{\left(-\frac{1}{2},-\frac{1}{2}\right)}(k)\left(\int_{-1}^1 \mathcal{J}_{k}^{\left(-\frac{1}{2},-\frac{1}{2}\right)}(x)\mathcal{J}_{m}^{\left(\frac{1}{2},-\frac{1}{2}\right)}(x)\varphi_{n,+}(x)(1-x)^{-1/2} (1+x)^{-1/2}  dx\times\right.\\
  \left.\int_{-1}^1 \mathcal{J}_{l}^{\left(\frac{1}{2},-\frac{1}{2}\right)}(y)\mathcal{J}_{k}^{\left(-\frac{1}{2},-\frac{1}{2}\right)}(y)\varphi_{n,-}(y)(1-y)^{-1/2} (1+y)^{-1/2}  dy\right) \\
  +\frac{1}{\pi}\sum_{k=0}^{\infty} W^{\left(-\frac{1}{2},-\frac{1}{2}\right)}(k)\mathcal{J}_{k}^{\left(-\frac{1}{2},-\frac{1}{2}\right)}(1)\varphi_{n,-}(1)\int_{-1}^1 \mathcal{J}_{k}^{\left(-\frac{1}{2},-\frac{1}{2}\right)}(x)\mathcal{J}_{m}^{\left(\frac{1}{2},-\frac{1}{2}\right)}(x)\varphi_{n,+}(x)(1-x)^{-1/2} (1+x)^{-1/2}  dx.
\end{multline}
Apply Lemma \ref{Lemma: 2} (3) to the second summation in \eqref{eq: 1},
\begin{multline*}
  \frac{1}{\pi}\sum_{k=0}^{\infty} W^{\left(-\frac{1}{2},-\frac{1}{2}\right)}(k)\mathcal{J}_{k}^{\left(-\frac{1}{2},-\frac{1}{2}\right)}(1)\varphi_{n,-}(1)\int_{-1}^1 \mathcal{J}_{k}^{\left(-\frac{1}{2},-\frac{1}{2}\right)}(x)\mathcal{J}_{m}^{\left(\frac{1}{2},-\frac{1}{2}\right)}(x)\varphi_{n,+}(x)(1-x)^{-1/2} (1+x)^{-1/2}  dx   \\
  =\mathcal{J}_{m}^{\left(\frac{1}{2},-\frac{1}{2}\right)}(1)\varphi_{n,\pm}(1).  
\end{multline*}
For the first summation, we apply Lemma \ref{Lemma: 2} (6),
\begin{multline*}
     \frac{1}{\pi^2} \sum_{k=0}^{\infty}   W^{\left(-\frac{1}{2},-\frac{1}{2}\right)}(k)\left(\int_{-1}^1 \mathcal{J}_{k}^{\left(-\frac{1}{2},-\frac{1}{2}\right)}(x)\mathcal{J}_{m}^{\left(\frac{1}{2},-\frac{1}{2}\right)}(x)\varphi_{n,+}(x)(1-x)^{-1/2} (1+x)^{-1/2}  dx\times\right.\\
  \left.\int_{-1}^1 \mathcal{J}_{l}^{\left(\frac{1}{2},-\frac{1}{2}\right)}(y)\mathcal{J}_{k}^{\left(-\frac{1}{2},-\frac{1}{2}\right)}(y)\varphi_{n,-}(y)(1-y)^{-1/2} (1+y)^{-1/2}  dy\right)\\
  =\frac{1}{\pi}\int_{-1}^1\mathcal{J}_{m}^{\left(\frac{1}{2},-\frac{1}{2}\right)}(y)\mathcal{J}_{l}^{\left(\frac{1}{2},-\frac{1}{2}\right)}(y)\varphi_{n,\pm}(y)(1-y)^{-1/2}(1+y)^{-1/2}  dy. 
\end{multline*}
   When $  m\ge 0$ and  $\ l=\text{virt}$, the proof is the same.    
\end{proof}
 
In the following, we omit computation of convolutions if it's similar to that of Proposition \ref{Prop: 1}. 

It is beneficial to write the results in Proposition \ref{Prop: 1} as contour integrals.  Note that 
\begin{equation*}
      \mathcal{J}_{m}^{\left(\pm\frac{1}{2},-\frac{1}{2}\right)}(x)\varphi_{n,\pm}(x)=\frac{1}{2\pi i}\oint \frac{\mathcal{J}_{m}^{\left(\pm\frac{1}{2},-\frac{1}{2}\right)}(u)\varphi_{n,\pm}(u)}{u-x} du,
\end{equation*}
where the $u$-contour  is a positively oriented simple loop that encircles point $x$. Then, if we enlarge the $u$-contour such that it encircles the interval $[-1,1]$,
\begin{multline*}
   \mathcal{T}^{n,-} \ast\phi_{n-1,+}^{n,-}(m,l)=\\
  \left\{\begin{array}{lc} 
   \frac{W^{\left(\frac{1}{2},-\frac{1}{2}\right)}(l)}{\pi}\int_{-1}^1\mathcal{J}_{l}^{\left(\frac{1}{2},-\frac{1}{2}\right)}(x)\frac{1}{2\pi i}\oint  \frac{\mathcal{J}_{m}^{\left(-\frac{1}{2},-\frac{1}{2}\right)}(u)\varphi_{n,+}(u)}{(u-1)(u-x)} 
    (1-x)^{1/2} (1+x)^{-1/2} du dx,         &  m,l \ge 0,\\
  \frac{1}{2\pi i}\oint \frac{\mathcal{J}_{m}^{\left(\frac{1}{2},-\frac{1}{2}\right)}(u)\varphi_{n,+}(u)}{u-1} du,    &  m\ge 0, \ l=\text{virt}.
  \end{array}\right.   
\end{multline*}

\begin{multline*}
   \mathcal{T}^{n,+} \ast\phi_{n,-}^{n,+}\ast \mathcal{T}^{n,-} \ast\phi_{n-1,+}^{n,-}(m,l)=\\
  \left\{    \begin{array}{lc}
  \frac{1}{\pi}\int_{-1}^1\mathcal{J}_{l}^{\left(\frac{1}{2},-\frac{1}{2}\right)}(x)\frac{1}{2\pi i}\oint  \frac{\mathcal{J}_{m}^{\left(\frac{1}{2},-\frac{1}{2}\right)}(u)\varphi_{n,\pm}(u)}{(u-1)(u-x)} 
    (1-x)^{1/2} (1+x)^{-1/2} du dx,         &  m,l \ge 0,\\
 \frac{1}{2\pi i} \oint \frac{\mathcal{J}_{m}^{\left(\frac{1}{2},-\frac{1}{2}\right)}(u)\varphi_{n,\pm}(u)}{u-1} du,   &  m\ge 0,\ l=\text{virt}.\\
    \end{array}\right. 
\end{multline*}

Following the same computation, we compute iterative convolutions of the functions.
\begin{Proposition}\label{Prop: 5}
For $k\ge 0$,
\begin{multline*}
  \mathcal{T}^{n,+} \ast\phi_{-}^{+}\ast\cdots\ast \mathcal{T}^{n-k,-} \ast\phi_{+}^{-}(m,l)=\\
  \left\{    \begin{array}{lc}
  \begin{split}
      \frac{1}{\pi}\int_{-1}^1\mathcal{J}_{l}^{\left(\frac{1}{2},-\frac{1}{2}\right)}(x) \frac{1}{(2\pi i)^{k+1}}\oint\cdots\oint
  \frac{H^{n,k}_m\left(u_n,u_{n-1},\ldots,u_{n-k}\right)}{(u_{n-k}-x)} du_n\cdots du_{n-k}\\
  \times(1-x)^{1/2}(1+x)^{-1/2} dx, 
  \end{split}
 &  m\ge 0,\ l\ge 0,   \\
  \frac{1}{(2\pi i)^{k+1}} \oint\cdots\oint H^{n,k}_m(u_n,u_{n-1},\ldots,u_{n-k}) du_n\cdots du_{n-k}, &  m\ge 0,\ l=\text{virt},\\
    \end{array}\right.   
\end{multline*}
where
\begin{equation*}
 H^{n,k}_m\left(u_n,u_{n-1},\ldots,u_{n-k+1}\right)=   \frac{J_{m}^{\left(\frac{1}{2},-\frac{1}{2}\right)}(u_{n})\varphi_{n,\pm}(u_{n})\varphi_{n-1,\pm}(u_{n-1})\cdots\varphi_{n-k,\pm}(u_{n-k})}{(u_{n}-1)(u_n-u_{  n-1})(u_{n-1}-1)(u_{n-1}-u_{n-2})\cdots(u_{n-k}-1)}
\end{equation*}
and the $u_i$-contour  is a positively oriented simple loop that encircles point $1$ together with the $u_{i-1}$-contour when $n-k+1\le i\le n$, the $u_{n-k}$-contour  is a positively oriented simple loop that encircles the interval $[-1,1]$.
\end{Proposition}

Now we are ready to compute $  \phi^{t^{2n_2-1/2+a_2}_{b_2},t^{2n_1-1/2+a_1}_{b_1}}$.
\begin{Proposition} \label{Prop: 3}
Suppose  $t^{2n_1-1/2+a_1}_{b_1}> t^{2n_2-1/2+a_2}_{b_2}$, if $n_1=n_2$, 
\begin{multline*}
 \phi^{t^{2n_2-1/2+a_2}_{b_2},t^{2n_1-1/2+a_1}_{b_1}}(s_1,s_2)=\frac{W^{\left(a_1,-\frac{1}{2}\right)}(s_1)}{\pi}\int_{-1}^1\mathcal{J}_{s_1}^{\left(a_1,-\frac{1}{2}\right)}(x)J_{s_2}^{\left(a_2,-\frac{1}{2}\right)}(x)\times\\ \varphi_{t^{2n_1-1/2+a_1}_{b_1}-t^{2n_2-1/2+a_2}_{b_2}} (x)(1-x)^{a_1}(1+x)^{-1/2}  dx;    
\end{multline*}
if $n_1<n_2$,
\begin{multline*}
        \phi^{t^{2n_2-1/2+a_2}_{b_2},t^{2n_1-1/2+a_1}_{b_1}}(s_1,s_2)\\      =\frac{W^{\left(a_1,-\frac{1}{2}\right)}(s_1)}{\pi}\int_{-1}^1\mathcal{J}_{s_1}^{\left(a_1,-\frac{1}{2}\right)}(x) g^{a_2}_{n_1+1,n_2}(x)
  \varphi_{t^{2n_1-1/2+a_1}_{b_1}-t^{2n_1}_{0}} (x)(1-x)^{a_1}(1+x)^{-1/2}  dx,
\end{multline*}
where
  \begin{multline*}
    g^{a_2}_{n_1+1,n_2}(x)=\frac{1}{(2\pi i)^{n_2-n_1}}\times\\
      \oint\cdots\oint \frac{J_{s_2}^{\left(a_2,-\frac{1}{2}\right)}(u_{n_2})\varphi_{t^{2n_2-1}_{c(2n_2-1)}-t^{2n_2-1/2+a_2}_{b_2}}(u_{n_2}) \varphi_{n_2-1,\pm}(u_{n_2-1})\cdots\varphi_{n_1+1,\pm}(u_{n_1+1})}{(u_{n_2}-1)(u_{  n_2}-u_{n_2-1})\cdots(u_{n_1+1}-1)(u_{n_1+1}-x)} du_{n_2}\cdots du_{n_1+1}.   
\end{multline*}
\end{Proposition}   

\begin{proof}
Convoluting the leftover terms with the expression for $\mathcal{T}^{n,+} \ast\phi_{n,-}^{n,+}\ast\cdots\ast \mathcal{T}^{n-k,-} \ast\phi_{n-k-1,+}^{n-k,-}$ in Proposition \ref{Prop: 5} finishes the proof.
\end{proof}

\subsection{The matrix $M^{n,a}$ and $\psi_{k-l}^{t^{2k-1/2+a_1}_b,N}$}

\begin{Proposition}

When $k \ge l$,
\begin{multline*}
  \psi_{k-l}^{ t^{2k-1/2+a_1}_b,N}(s)=
  \frac{W^{\left(a_1,-\frac{1}{2}\right)}(s)}{\pi}\int_{-1}^1E^\omega(y)\mathcal{J}_{s}^{(a_1,-\frac{1}{2})}(y)
  \varphi_{t^{2k-1/2+a_1}_b-t^{N}_{0}}(y)\times\\
  (y-1)^{k-l} (1-y)^{a_1}(1+y)^{-1/2}  dy,     
\end{multline*}

When $k<l$,
\begin{multline*}
  \psi_{k-l}^{ t^{2k-1/2+a_1}_b,N}(s)=
     \frac{W^{\left(a_1,-\frac{1}{2}\right)}(s)}{\pi}\frac{1}{(2\pi i)^{l-k}}\int_{-1}^1\mathcal{J}_{s}^{(a_1,-\frac{1}{2})}(x)\times\\
 \oint\cdots\oint \frac{E^\omega(u_l)\varphi_{t^{2l-1}_{c(2l-1)}-t^{N}_{0}}(u_{l})\varphi_{l-1,\pm}(u_{l-1})\cdots\varphi_{k+1,\pm}(u_{k+1}) }{(u_{l}-1)(u_{l}-u_{l-1})\cdots(u_{k+1}-1)(u_{k+1}-x)} du_{l}\cdots du_{k+1}\times \\
   \varphi_{t^{2k-1/2+a_1}_b-t^{2k}_{0}}(x)  (1-x)^{a_1}(1+x)^{-1/2}  dx, 
\end{multline*}
 where the $u_i$-contour ($k+2\le i\le l$)  is a positively oriented simple loop that encircles point $1$ and $u_{i-1}$-contour, the $u_{k+1}$-contour  is a positively oriented simple loop that encircles interval $[-1,1]$.
\end{Proposition}
\begin{proof}
The proof can be done by direct calculation of convolutions, or by the induction method as in the proof of Theorem 4.4 of \cite{borodin2011random}.
\end{proof}

\begin{Lemma}
The matrix $M^{n,a}$ is upper triangular and invertible.
\end{Lemma}
\begin{proof}

We first show that when $k>l$, $ M_{k,l}^{n,a}=0$.
Recall that
\begin{equation*}
   M_{k,l}^{n,a}= \psi^{n,a}_{n-l} \ast \mathcal{T}^{n,a} \ast\phi_{-a}^{a}\ast    \mathcal{T}^{n-1/2+a,-a}\ast \cdots\ast \mathcal{T}^{k,-}\ast\phi_{+}^{-}(\text{virt}).
\end{equation*}

So,
\begin{equation*}
   M_{k,l}^{n,a}=  \frac{1}{(2\pi i)^{n-k+1}} \oint\cdots\oint \frac{E^\omega(u_n)(u_n-1)^{n-l}\varphi_{n,a}(u_N)\cdots \varphi_{k,\pm}(u_{k})}{(u_n-1)(u_n-u_{n-1})\cdots(u_{k}-1)} du_n\cdots du_{k}=0.
\end{equation*}
When $k\le l$,
\begin{equation*}
   M_{k,l}^{n,a}=  \frac{1}{(2\pi i)^{l-k+1}} \oint\cdots\oint \frac{E^\omega(u_l)\varphi_{t^{2l-1}_{c(2l-1)}-t^{2N-1/2+a}_{0}}(u_{l})\varphi_{l-1,\pm}(u_{l-1})\cdots \varphi_{k,\pm}(u_k)}{(u_{l}-1)(u_l-u_{l-1})\cdots(u_k-1)} du_l\cdots du_k.
\end{equation*}
Thus the diagonal elements of $M^{n,a}$ are nonzero and $M^{n,a}$ is invertible.
\end{proof}

\subsection{$ \Phi_{k-j}^{ t^{2k-1/2+a_1}_b,N}(s)$ and $\displaystyle {\sum_{k=1}^{n_2}\psi_{n_1-k}^{t^{2n_1-1/2+a_1}_{b_1},N}(s_1)\Phi_{n_2-k}^{ t^{2n_2-1/2+a_2}_{b_2},N}(s_2)}$ }

\begin{Lemma}
For any $1\le j\le  k\le N$, define
\begin{equation*}
    \Phi_{k-j}^{ t^{2k-1/2+a_1}_b,N}(s)=\frac{1}{2\pi i}\oint\frac{J_{s}^{\left(a_1,-\frac{1}{2}\right)}(u)}{E^\omega(u)\varphi_{t^{2k-1/2+a_1}_b-t^{N}_{0}}(u)}\frac{1}{(u-1)^{k-j+1}}du.
\end{equation*}

Then,
\begin{enumerate}[{(1)}]
    \item $\left\{\Phi_{k-j}^{ t^{2k-1/2+a_1}_b,N}(s)\right\}_{j=1}^k$ is a  basis of the linear span of 
    \begin{equation*}
        \left\{\mathcal{T}_{t^{2k-1/2+a_1}_b,t^{2k-1/2+a_1}_{c\left(2k-1/2+a_1\right)}}^{k,a_1}\ast\cdots\ast\mathcal{T}_{t^{2j-1}_0,t^{2j-1}_{c\left(2j-1\right)}}^{j,-}\ast \phi_{+}^{-}(s,\text{virt})\right\}_{j=1}^k.
    \end{equation*}
    \item For $0\le j_1,j_2 \le k-1$, 
    \begin{equation*}
        \sum_{s\ge 0} \Phi_{j_1}^{t^{2k-1/2+a_1}_b,N}(s)\psi_{j_2}^{t^{2k-1/2+a_1}_b,N}(s)=\delta_{j_1,j_2}.
    \end{equation*}
\end{enumerate}
\end{Lemma}
\begin{proof}
To prove (1), we use the fact that $ \left\{\frac{\partial^j}{\partial u^j}J_{s}^{\left(a_1,-\frac{1}{2}\right)}(u)\Big|_{u=1}\right\}_{j=0}^{k-1}$ is a set of polynomials in variable $s$ of degree $2j+1/2+a_1$, which is a linear basis of both the linear span of 
    \begin{equation*}
      \left\{\mathcal{T}_{t^{2k-1/2+a_1}_b,t^{2k-1/2+a_1}_{c\left(2k-1/2+a_1\right)}}^{k,a_1}\ast\cdots\ast\mathcal{T}_{t^{2j-1}_0,t^{2j-1}_{c\left(2j-1\right)}}^{j,-}\ast \phi_{+}^{-}(s,\text{virt})\right\}_{j=1}^k
    \end{equation*}
and    $\left\{\Phi_{k-j}^{ t^{2k-1/2+a_1}_b,N}(s)\right\}_{j=1}^k$.\\
(2) follows from Lemma \ref{Lemma: 2} (6).
\end{proof}
Next, we calculate $\sum_{k=1}^{n_2}\psi_{n_1-k}^{t^{2n_1-1/2+a_1}_{b_1},N}(s_1)\Phi_{n_2-k}^{ t^{2n_2-1/2+a_2}_{b_2},N}(s_2)$.
When $n_1<n_2$, let 
\begin{equation*}
   f^N_{t^{2n_2-1/2+a_2}_{b_2}}(u)=\frac{J_{s_2}^{\left(a_2,-\frac{1}{2}\right)}(u)}{E^\omega(u)\varphi_{t^{2n_2-1/2+a_2}_{b_2}-t^{N}_{0}}(u)} 
\end{equation*}
and
\begin{equation*}
 h^N_{n_1+1,k}(x)=\frac{1}{(2\pi i)^{k-n_1}}\oint\cdots\oint \frac{E^\omega(u_{k})\varphi_{t^{2k-1}_{c(2k-1)}-t^{N}_{0}}(u_{k})\cdots \varphi_{n_1+1,\pm}(u_{n_1+1})}{(u_{k}-1)(u_{k}-u_{k-1})\cdots(u_{n_1+1}-1)(u_{n_1+1}-x)} du_{k}\cdots du_{n_1+1}.
\end{equation*}
Then
\begin{multline*}
    g^{a_2}_{n_1+1,n_2}(x)=\frac{1}{(2\pi i)^{n_2-n_1}}\times \\
     \oint\cdots\oint \frac{f^N_{t^{2n_2-1/2+a_2}_{b_2}}(u_{n_2})E^\omega(u_{n_2})\varphi_{t^{2n_2-1}_{c(2n_2-1)}-t^{N}_{0}}(u_{n_2}) \cdots\varphi_{n_1+1,\pm}(u_{n_1+1})}{(u_{n_2}-1)(u_{  n_2}-u_{n_2-1})\cdots(u_{n_1+1}-1)(u_{n_1+1}-x)} du_{n_2}\cdots du_{n_1+1},
\end{multline*}
\begin{equation*}
  \Phi_{n_2-k}^{ t^{2n_2-1/2+a_2}_b,N} = \frac{1}{2\pi i}\oint \frac{f^N_{t^{2n_2-1/2+a_2}_{b_2}}(u)}{(u-1)^{n_2-k+1}}du.
\end{equation*}

\begin{Lemma}\label{Lemma: 3} When $n_1< n_2$,
\begin{multline*}
    \sum_{k=n_1+1}^{n_2}h^N_{n_1+1,k}(x)* \Phi_{n_2-k}^{ t^{2n_2-1/2+a_2}_b,N}
    =\frac{1}{2\pi i}\oint \frac{f^N_{t^{2n_2-1/2+a_2}_{b_2}}(u)}{(u-1)^{n_2-n_1}(x-u)}duE^\omega(x)\varphi_{t^{2n_1+1}_{c(2n_1+1)}-t^{N}_{0}}(x)+g^{a_2}_{n_1+1,n_2}(x).\\
 \end{multline*}  
\end{Lemma}
\begin{proof}
First, when $n_2=n_1+1$, the equality holds.
Now suppose this holds for fixed $n_1$ and $n_2$, we show this holds for $n_2$ and $n_1-1$ as well.  
Then, it suffices to show that 

\begin{multline*}
  \frac{1}{2\pi i}\oint \frac{E^\omega(u_{n_1})\varphi_{t^{2n_1-1}_{c(2n_1-1)}-t^{N}_{0}}(u_{n_1})}{(u_{n_1}-1)(u_{n_1}-x)}du_{n_1}\cdot \Phi_{n_2-n_1}^{ t^{2n_2-1/2+a_2}_b,N}   \\
    +\sum_{k=n_1+1}^{n_2}\oint \frac{h^N_{n_1,k}(u_{n_1})\varphi_{n_1,\pm}(u_{n_1})}{(u_{n_1}-1)(u_{n_1}-x)}du_{n_1}\cdot \Phi_{n_2-k}^{ t^{2n_2-1/2+a_2}_b,N} \\  
     =\frac{1}{2\pi i}\oint \frac{f^N_{t^{2n_2-1/2+a_2}_{b_2}}(u)}{(u-1)^{n_2-n_1+1}(x-u)}duE^\omega(x)\varphi_{t^{2n_1-1}_{c(2n_1-1)}-t^{N}_{0}}(x)+\frac{1}{2\pi i}\oint\frac{g^{a_2}_{n_1+1,n_2}(u_{n_1})\varphi_{n_1}(u_{n_1})}{(u_{n_1}-1)(u_{n_1}-x)}du_{n_1}.  
\end{multline*}

By induction assumption,
\begin{multline*}
     \sum_{k=n_1+1}^{n_2}\left(\oint \frac{h^N_{n_1,k}(u_{n_1})\varphi_{n_1,\pm}(u_{n_1})}{(u_{n_1}-1)(u_{n_1}-x)}du_{n_1}\cdot
   \Phi_{n_2-k}^{ t^{2n_2-1/2+a_2}_b,N}  \right)  -\frac{1}{2\pi i}\oint\frac{g^{a_2}_{n_1+1,n_2}(u_{n_1})\varphi_{n_1,\pm}(u_{n_1})}{(u_{n_1}-1)(u_{n_1}-x)}du_{n_1}  \\
    =\frac{\varphi_{n_1,\pm}(1)(-g^{a_2}_{n_1+1,n_2}(1)+\sum_{k=n_1+1}^{n_2}h^N_{n_1,k}(1) \Phi_{n_2-k}^{ t^{2n_2-1/2+a_2}_b,N})}{1-x}\\
    -\frac{\varphi_{n_1,\pm}(x)(-g^{a_2}_{n_1+1,n_2}(x)+\sum_{k=n_1+1}^{n_2}h^N_{n_1,k}(x) \Phi_{n_2-k}^{ t^{2n_2-1/2+a_2}_b,N})}{1-x}\\
    =-\frac{1}{1-x}\cdot\frac{1}{2\pi i}\oint \frac{f^N_{t^{2n_2-1/2+a_2}_{b_2}}(u)}{(u-1)^{n_2-n_1+1}}duE^\omega(1)\varphi_{t^{2n_1-1}_{c(2n_1-1)}-t^{N}_{0}}(1)\\
    -\frac{1}{1-x}\cdot\frac{1}{2\pi i}\oint \frac{f^N_{t^{2n_2-1/2+a_2}_{b_2}}(u)}{(u-1)^{n_2-n_1}(x-u)}duE^\omega(x)\varphi_{t^{2n_1-1}_{c(2n_1-1)}-t^{N}_{0}}(x),  
\end{multline*}
while
\begin{multline*}
      \frac{1}{2\pi i}\oint \frac{E^\omega(u_{n_1})\varphi_{t^{2n_1-1}_{c(2n_1-1)}-t^{N}_{0}}(u_{n_1})}{(u_{n_1}-1)(u_{n_1}-x)}du_{n_1}\cdot \Phi_{n_2-n_1}^{ t^{2n_2-1/2+a_2}_b,N}\\
      -\frac{1}{2\pi i}\oint \frac{f^N_{t^{2n_2-1/2+a_2}_{b_2}}(u)}{(u-1)^{n_2-n_1+1}(x-u)}duE^\omega(x)\varphi_{t^{2n_1-1}_{c(2n_1-1)}-t^{N}_{0}}(x)\\
      =\frac{E^\omega(1)\varphi_{t^{2n_1-1}_{c(2n_1-1)}-t^{N}_{0}}(1)-E^\omega(x)\varphi_{t^{2n_1-1}_{c(2n_1-1)}-t^{N}_{0}}(x)}{1-x}\cdot\frac{1}{2\pi i}\oint \frac{f^N_{t^{2n_2-1/2+a_2}_{b_2}}(u)}{(u-1)^{n_2-n_1+1}}du\\
    -  \frac{1}{2\pi i}\oint \frac{f^N_{t^{2n_2-1/2+a_2}_{b_2}}(u)}{(u-1)^{n_2-n_1+1}(x-u)}duE^\omega(x)\varphi_{t^{2n_1-1}_{c(2n_1-1)}-t^{N}_{0}}(x)\\
    =\frac{1}{1-x}\cdot \frac{1}{2\pi i}\oint \frac{f^N_{t^{2n_2-1/2+a_2}_{b_2}}(u)}{(u-1)^{n_2-n_1+1}}duE^\omega(1)\varphi_{t^{2n_1-1}_{c(2n_1-1)}-t^{N}_{0}}(1)\\
    +\frac{1}{1-x}\cdot\frac{1}{2\pi i}\oint \frac{ f^N_{t^{2n_2-1/2+a_2}_{b_2}}(u)}{(u-1)^{n_2-n_1}(x-u)}duE^\omega(x)\varphi_{t^{2n_1-1}_{c(2n_1-1)}-t^{N}_{0}}(x).
\end{multline*}
\end{proof}

\begin{Proposition}\label{Prop: 2}
      
If $n_1\ge n_2\ge 1$ and     and $s_1,s_2\in \mathbb{Z}_{\ge 0}$, we have 
 \begin{multline*}
 \sum_{k=1}^{n_2}\psi_{n_1-k}^{t^{2n_1-1/2+a_1}_{b_1},N}(s_1)\Phi_{n_2-k}^{ t^{2n_2-1/2+a_2}_{b_2},N}(s_2)= \\
 \frac{W^{\left(a_1,-\frac{1}{2}\right)}(s_1)}{\pi}\frac{1}{2\pi i}\int_{-1}^1\oint\mathcal{J}_{s_1}^{\left(a_1,-\frac{1}{2}\right)}(y){J}_{s_2}^{\left(a_2,-\frac{1}{2}\right)}(u)
  \frac{E^\omega(y)}{E^\omega(u)}
  \frac{
  \varphi_{t^{2n_1-1/2+a_1}_{b_1}-t^{N}_{0}}(y)}{\varphi_{t^{2n_2-1/2+a_2}_{b_2}-t^{N}_{0}}(u)}\times\\
  \frac{(y-1)^{n_1}}{(u-1)^{n_2}}
  \frac{(1-y)^{a_1}(1+y)^{-\frac{1}{2}}}{y-u} dudy\\
    +\frac{W^{\left(a_1,-\frac{1}{2}\right)}(s_1)}{\pi}\int_{-1}^1\mathcal{J}_{s_1}^{\left(a_1,-\frac{1}{2}\right)}(y){J}_{s_2}^{\left(a_2,-\frac{1}{2}\right)}(y)\frac{
  \varphi_{t^{2n_1-1/2+a_1}_{b_1}-t^{N}_{0}}(y)}{\varphi_{t^{2n_2-1/2+a_2}_{b_2}-t^{N}_{0}}(y)}
  (y-1)^{n_1-n_2}
  (1-y)^{a_1}(1+y)^{-\frac{1}{2}}  dy.
 \end{multline*}

If $1\le n_1<n_2$,
 \begin{multline*}
 \sum_{k=1}^{n_2}\psi_{n_1-k}^{t^{2n_1-1/2+a_1}_{b_1},N}(s_1)\Phi_{n_2-k}^{ t^{2n_2-1/2+a_2}_{b_2},N}(s_2)= \\
  \frac{W^{\left(a_1,-\frac{1}{2}\right)}(s_1)}{\pi}\frac{1}{2\pi i}\int_{-1}^1\oint\mathcal{J}_{s_1}^{\left(a_1,-\frac{1}{2}\right)}(y)\mathcal{J}_{s_2}^{\left(a_2,-\frac{1}{2}\right)}(u)
 \frac{E^\omega(y)}{E^\omega(u)}\frac{
  \varphi_{t^{2n_1-1/2+a_1}_{b_1}-t^{N}_{0}}(y)}{\varphi_{t^{2n_2-1/2+a_2}_{b_2}-t^{N}_{0}}(u)}
 \frac{(y-1)^{n_1}}{(u-1)^{n_2}}
  \frac{(1-y)^{a_1}(1+y)^{-\frac{1}{2}}}{y-u} dudy\\
  +\phi^{t^{2n_2-1/2+a_2}_{b_2},t^{2n_1-1/2+a_1}_{b_1}}(s_1,s_2).
 \end{multline*}
\end{Proposition}
\begin{proof}
The calculation for \begin{equation*}
    \sum_{k=1}^{\text{min}\{n_1,n_2\}}\psi_{n_1-k}^{t^{2n_1-1/2+a_1}_{b_1},N}(s_1)\Phi_{n_2-k}^{ t^{2n_2-1/2+a_2}_{b_2},N}(s_2)
\end{equation*} follows the same arguments as the proof for Proposition 4.6 in \cite{borodin2011random}. We only show the proof when $1\le n_1<n_2$.

First, 
\begin{multline*}
  \sum_{k=1}^{n_1}\psi_{n_1-k}^{t^{2n_1-1/2+a_1}_{b_1},N}(s_1)\Phi_{n_2-k}^{ t^{2n_2-1/2+a_2}_{b_2},N}(s_2)= \\
  \frac{W^{\left(a_1,-\frac{1}{2}\right)}(s_1)}{\pi}\frac{1}{2\pi i}\int_{-1}^1\oint\mathcal{J}_{s_1}^{\left(a_1,-\frac{1}{2}\right)}(y)\mathcal{J}_{s_2}^{\left(a_2,-\frac{1}{2}\right)}(u)\frac{E^\omega(y)}{E^\omega(u)}
  \frac{
  \varphi_{t^{2n_1-1/2+a_1}_{b_1}-t^{N}_{0}}(y)}{\varphi_{t^{2n_2-1/2+a_2}_{b_2}-t^{N}_{0}}(u)}\times\\
  \frac{(y-1)^{n_1}}{(u-1)^{n_2}}\left(1-\left(\frac{u-1}{y-1}\right)^{n_1}\right)
  \frac{(1-y)^{a_1}(1+y)^{-\frac{1}{2}}}{y-u} dudy.
\end{multline*}

Now we only need to find $ \sum_{k=n_1+1}^{n_2}\psi_{n_1-k}^{t^{2n_1-1/2+a_1}_{b_1},N}(s_1)\Phi_{n_2-k}^{ t^{2n_2-1/2+a_2}_{b_2},N}(s_2)$.
By Lemma \ref{Lemma: 3},
\begin{multline*}
  \sum_{k=n_1+1}^{n_2}\psi_{n_1-k}^{t^{2n_1-1/2+a_1}_{b_1},N}(s_1)\Phi_{n_2-k}^{ t^{2n_2-1/2+a_2}_{b_2},N}(s_2)\\
   =\frac{W^{\left(a_1,-\frac{1}{2}\right)}(s_1)}{\pi}\int_{-1}^1\mathcal{J}_{s_1}^{\left(a_1,-\frac{1}{2}\right)}(x)\left(\sum_{k=n_1+1}^{n_2}h^N_{n_1+1,k}(x)*\Phi_{n_2-k}^{ t^{2n_2-1/2+a_2}_b,N}\right)\varphi_{t^{2n_1-1/2+a_1}_b-t^{2n_1}_{0}}(x)  (1-x)^{a_1}(1+x)^{-1/2}  dx  \\
  = \frac{W^{\left(a_1,-\frac{1}{2}\right)}(s_1)}{\pi(2\pi i)} \int_{-1}^1\oint \frac{\mathcal{J}_{s_1}^{\left(a_1,-\frac{1}{2}\right)}(x) f^N_{t^{2n_2-1/2+a_2}_{b_2}}(u)}{(u-1)^{n_2-n_1}(x-u)}E^\omega(x)\varphi_{t^{2n_1-1/2+a_1}_{b_1}-t^{N}_{0}}(x)(1-x)^{a_1}(1+x)^{-1/2}  dudx \\
   +\frac{W^{\left(a_1,-\frac{1}{2}\right)}(s_1)}{\pi}\int_{-1}^1\mathcal{J}_{s}^{\left(a_1,-\frac{1}{2}\right)}(x)g^{a_2}_{n_1+1,n_2}(x)\varphi_{t^{2n_1-1/2+a_1}_{b_1}-t^{2n_1}_{0}}(x)  (1-x)^{a_1}(1+x)^{-1/2}  dx. \end{multline*}
   Note that 
   \begin{multline*}
       \frac{W^{\left(a_1,-\frac{1}{2}\right)}(s_1)}{\pi}\int_{-1}^1\mathcal{J}_{s_1}^{\left(a_1,-\frac{1}{2}\right)}(x)g^{a_2}_{n_1+1,n_2}(x)\varphi_{t^{2n_1-1/2+a_1}_{b_1}-t^{2n_1}_{0}}(x)  (1-x)^{a_1}(1+x)^{-1/2}  dx\\
       =\phi^{t^{2n_2-1/2+a_2}_{b_2},t^{2n_1-1/2+a_1}_{b_1}}(s_1,s_2),
   \end{multline*}
   Thus,
   \begin{multline*}
   \sum_{k=n_1+1}^{n_2}\psi_{n_1-k}^{t^{2n_1-1/2+a_1}_{b_1},N}(s_1)\Phi_{n_2-k}^{ t^{2n_2-1/2+a_2}_{b_2},N}(s_2)   \\
   = \frac{W^{\left(a_1,-\frac{1}{2}\right)}(s_1)}{\pi(2\pi i)} \int_{-1}^1\oint \frac{\mathcal{J}_{s_1}^{\left(a_1,-\frac{1}{2}\right)}(x) f^N_{t^{2n_2-1/2+a_2}_{b_2}}(u)}{(u-1)^{n_2-n_1}(x-u)}E^\omega(x)\varphi_{t^{2n_1-1/2+a_1}_{b_1}-t^{N}_{0}}(x)(1-x)^{a_1}(1+x)^{-1/2}  dudx \\
    +\phi^{t^{2n_2-1/2+a_2}_{b_2},t^{2n_1-1/2+a_1}_{b_1}}(s_1,s_2).   
   \end{multline*}

 Adding the above two summations together, we get
 \begin{multline*}
 \sum_{k=1}^{n_2}\psi_{n_1-k}^{t^{2n_1-1/2+a_1}_{b_1},N}(s_1)\Phi_{n_2-k}^{ t^{2n_2-1/2+a_2}_{b_2},N}(s_2)= \\
  \frac{W^{\left(a_1,-\frac{1}{2}\right)}(s_1)}{\pi}\frac{1}{2\pi i}\int_{-1}^1\oint\mathcal{J}_{s_1}^{\left(a_1,-\frac{1}{2}\right)}(y)\mathcal{J}_{s_2}^{\left(a_2,-\frac{1}{2}\right)}(u)
 \frac{E^\omega(y)}{E^\omega(u)}\frac{
  \varphi_{t^{2n_1-1/2+a_1}_{b_1}-t^{N}_{0}}(y)}{\varphi_{t^{2n_2-1/2+a_2}_{b_2}-t^{N}_{0}}(u)}
 \frac{(y-1)^{n_1}}{(u-1)^{n_2}}
  \frac{(1-y)^{a_1}(1+y)^{-\frac{1}{2}}}{y-u} dudy\\
  +\phi^{t^{2n_2-1/2+a_2}_{b_2},t^{2n_1-1/2+a_1}_{b_1}}(s_1,s_2).
 \end{multline*}

\end{proof}

\subsection{Computing the kernel}
In this section, we apply Lemma \ref{Lemma: 1}  to derive the correlation kernel $K$. Adding Proposition \ref{Prop: 3} and Proposition \ref{Prop: 2}, we get the following:

When $(n_1,a_1,t_1)\succ (n_2,a_2,t_2)$, which means  $t_1\le t_2$,  $2n_1-1/2+a_1\ge  2n_2-1/2+a_2$ and $(n_1,a_1,t_1)\neq (n_2,a_2,t_2)$. Let $t^{2n_i-1/2+a_i}_{b_i}=t_i$, $i=1,2$. By the fact that $ \phi^{t_2,t_1}=0$ when $t_1\le t_2$, we have 
\begin{multline*}
  K(n_1,a_1,t_1,s_1;n_2,a_2,t_2,s_2)= \sum_{k=1}^{n_2}\psi_{n_1-k}^{t^{2n_1-1/2+a_1}_{b_1},N}(s_1)\Phi_{n_2-k}^{ t^{2n_2-1/2+a_2}_{b_2},N}(s_2) \\ =\frac{W^{\left(a_1,-\frac{1}{2}\right)}(s_1)}{\pi}\frac{1}{2\pi i}\int_{-1}^1\oint\mathcal{J}_{s_1}^{\left(a_1,-\frac{1}{2}\right)}(y){J}_{s_2}^{\left(a_2,-\frac{1}{2}\right)}(u)
 \frac{E^\omega(y)}{E^\omega(u)} \frac{e^{t_1(y-1)}}{e^{t_2(u-1)}}
  \frac{(y-1)^{n_1}}{(u-1)^{n_2}}
  \frac{(1-y)^{a_1}(1+y)^{-\frac{1}{2}}}{y-u} dudy\\
    +\frac{W^{\left(a_1,-\frac{1}{2}\right)}(s_1)}{\pi}\int_{-1}^1\mathcal{J}_{s_1}^{\left(a_1,-\frac{1}{2}\right)}(y){J}_{s_2}^{\left(a_2,-\frac{1}{2}\right)}(y)
  \frac{e^{t_1(y-1)}}{e^{t_2(y-1)}}
  (y-1)^{n_1-n_2}
  (1-y)^{a_1}(1+y)^{-\frac{1}{2}}  dy.
\end{multline*}
When $(n_1,a_1,t_1)\nsucc (n_2,a_2,t_2)$, which means $2n_1-1/2+a_1\le  2n_2-1/2+a_2$ and $t_1> t_2$,
\begin{multline*}
    K(n_1,a_1,t_1,s_1;n_2,a_2,t_2,s_2)=-\phi^{t^{2n_2-1/2+a_2}_{b_2},t^{2n_1-1/2+a_1}_{b_1}}(s_1,s_2)+\sum_{k=1}^{n_2}\psi_{n_1-k}^{t^{2n_1-1/2+a_1}_{b_1},N}(s_1)\Phi_{n_2-k}^{ t^{2n_2-1/2+a_2}_{b_2},N}(s_2)\\ =\frac{W^{\left(a_1,-\frac{1}{2}\right)}(s_1)}{\pi}\frac{1}{2\pi i}\int_{-1}^1\oint\mathcal{J}_{s_1}^{\left(a_1,-\frac{1}{2}\right)}(y)\mathcal{J}_{s_2}^{\left(a_2,-\frac{1}{2}\right)}(u)
 \frac{E^\omega(y)}{E^\omega(u)} \frac{e^{t_1(y-1)}}{e^{t_2(u-1)}}
  \frac{(y-1)^{n_1}}{(u-1)^{n_2}}
  \frac{(1-y)^{a_1}(1+y)^{-\frac{1}{2}}}{y-u} dudy.  
\end{multline*}

  \section{Non-commutative random walk on $U(\mathfrak{so}_{N+1})$}

In this section, we construct a non-commutative random walk on $U(\mathfrak{so}_{N+1})$, which is analogues to the non-commutative random walk on $U(\mathfrak{gl}_N)$ constructed in \cite{kuan2014threedimensional}.

We take the universal enveloping algebra of the Lie group $\mathfrak{so}_{N+1}$ as the state space and define a semi-group of the non-commutative Markov operator $\{P_t\}_{t\ge 0}$ on $U(\mathfrak{so}_{N+1})$.

For each class function $\kappa\in L^2(SO(N+1))$, we can define a state $\langle\cdot\rangle_{\kappa}$ on 
$U(\mathfrak{so}_{N+1})$ by $\langle X\rangle_{\kappa}=D(X)\kappa(U)|_{U=I}$, where $D$ is the canonical isomorphism from $U(\mathfrak{so}_{N+1})$ to the  algebra
of left–invariant differential operators on $SO(N+1)$ with complex coefficients (see e.g. \cite{DPZ}).

It is not hard to see (see e.g. \cite{kuan2014threedimensional}) that  if the class function $\kappa$ decomposes as 
\begin{equation*}
  \kappa=\sum_{\lambda}\hat{\kappa}(\lambda)\frac{\chi^\lambda}{\text{dim}\lambda},  
\end{equation*}
where $\lambda$ ranges over all irreducible representations of $SO(N+1)$ and $\chi^\lambda$ are the corresponding characters, then
\begin{equation}\label{eq: 9}
   \langle X\rangle_{\kappa}=\sum_\lambda\hat{\kappa}(\lambda)\sum_{i=1}^{\text{dim }\lambda}\text{Tr}\left(X|_{V_{\lambda}^i}\right). 
\end{equation}
In what follows, we let  $\kappa_t(O)=e^{t\text{Tr}(O-Id)}$  and  write $\langle\cdot\rangle_{t}$ for $\langle\cdot\rangle_{\kappa_t}$.

If $X=F_{i_1j_1}\cdots F_{i_kj_k}$, then the state can be computed with the following formula (see e.g. page 101 of \cite{V-harmonic}):
\begin{equation}\label{eq: 2.1}
  D(X)\kappa(O)=\partial_{t_1}\cdots\partial_{t_k}\kappa(Oe^{t_1F_{i_1j_1}}\cdots e^{t_k F_{i_kj_k}})|_{t_1=\cdots=t_k=0},
\end{equation}
where $e^{tF}$ is the usual exponential of matrices. In particular, we have 
\begin{equation*}
    e^{tF_{ij}}=\left\{
             \begin{array}{lr}
             Id+tF_{ij}, &  i\neq \pm j,\  i,j\neq 0, \\ 
              Id+tF_{ij}-\frac{t^2}{2}E_{-j,j},  & i=0,               \\
              Id+tF_{ij}-\frac{t^2}{2}E_{i,-i},  & j=0,               \\
           Id+(e^t-1)E_{ii}+(e^{-t}-1)E_{-i,-i},  &  i=j.
             \end{array}
\right.
\end{equation*}
Since \eqref{eq: 2.1} only involves linear terms in $t_j$, we can replace $e^{tF_{ij}}$ with $Id+tF_{ij}$. Applying Faa di Bruno formula (see e.g. \cite{kuan2014threedimensional}), we have
\begin{equation*}
    \langle F_{i_1,j_1}\cdots F_{i_mj_m}\rangle_t=\sum_{\pi\in \Pi} t^{|\pi|}\prod_{B\in \pi, B=\{b_1,\cdots,b_k\}}  \text{Tr}\left(\prod_{b\in B, B=\{b_1,\cdots,b_k\}}F_{i_bj_b}\right),
\end{equation*}
where $\Pi$ is the set of partitions of the set $\{1,2,\ldots,m\}$ and $B\in \pi$ means that $B$ is a block in partition $\pi$.

It is not hard to see that the non-commutative Markov operator $P_t$ defined in \cite{kuan2014threedimensional,Kuan-2+1} also defines a Markov operator on $U(\mathfrak{so}_{N+1})$.
\begin{Theorem}(Theorem 3.1 in \cite{kuan2014threedimensional})\label{Thm: 1}
Define  $P_t=(id\otimes \langle\cdot\rangle_{\kappa_t})\circ\Delta$, then 
\begin{enumerate}[{(1)}]
    \item $P_t$ satisfies the semi-group property $P_{t+s}=P_t\circ P_s$.
    \item $P_t$ preserves $Z(U(\mathfrak{so}_{N+1}))$, i.e. $P_t Z(U(\mathfrak{so}_{N+1}))\subset Z(U(\mathfrak{so}_{N+1}))$.
    \item For all $ Y\in Z(U(\mathfrak{so}_{N+1}))$,  $\langle P_tY\rangle_s=\langle Y\rangle_{s+t}$.
\end{enumerate}
\end{Theorem}
Since $P_t$ preserves $Z(U(\mathfrak{so}_{N+1}))$, we can expand $ P_t\Phi_{2k}^{N+1}$ in terms of generators of $Z(U(\mathfrak{so}_{N+1}))$.
\begin{Example}\label{exp: 2}
\begin{enumerate}[{(1)}]
    \item \begin{equation*}
    P_t\Phi_2^{N+1}=\Phi_2^{N+1}+\text{constant}. 
\end{equation*}
\item \begin{equation*}
    P_t\Phi_4^{N+1}=\Phi_4^{N+1}+16tn\Phi_2^{N+1}+\text{constant} .
\end{equation*}
\end{enumerate}

\end{Example}

 If we  define $Q^{n,a}_t$ to be the Markov operator for the point process $ \left\{x_k^{n,a}|\ 1\le k\le n,\ a\in\left\{\pm 1/2\right\},\ n\ge 1\right\}$ projected onto $\mathbb{Z}_{\ge 0}\times\{n\}\times\{a\}$, then $Q^{n,a}_t$ also defines a Markov operator on $Z(U(\mathfrak{so}_{N+1}))$ through the Harish-Chandra isomorphism.
\begin{Proposition}\label{Prop: 4}
For any $Y\in Z(U(\mathfrak{so}_{N+1}))$, there exists a polynomial $  \overline{p}^{N+1}$ such that 
\begin{equation*}
   \langle Y\rangle_{\frac{t}{2}}  =\mathbb{E}\left[\overline{p}^{N+1}_Y\left(x^{n,a}_1(t),\cdots,x^{n,a}_n(t)\right)\right].
\end{equation*}
In addition, $\langle Q^{n,a}_tY\rangle_{_{\frac{s}{2}}}=\langle Y\rangle_{_{\frac{s+t}{2}}}$. 
 \end{Proposition}
\begin{proof}
Suppose $Y\in Z(U(\mathfrak{so}_{N+1}))$ is sent to the polynomial $p_Y^{N+1}$ by the Harish-Chandra isomorphism. When $N+1=2n+1$, by \eqref{eq: 7}, for any $O\in SO(N+1)$, 
\begin{equation*}
  e^{\frac{1}{2}t\text{Tr}(O-I)}=\sum_{\lambda}\text{Prob}\left(x_k^{n,\frac{1}{2}}(t)=\lambda_k+n-k, 1\le k\le n\right)\frac{\chi^\lambda_{SO(2n+1)}(O)}{\text{dim}_{SO(2n+1)} \lambda}.
\end{equation*}
Recall that  $l_k=\lambda_k-k+\frac{1}{2}$ when $N+1=2n+1$, so  we define a polynomial $\overline{p}^{N+1}_Y$  with a change of variables, i.e. $\overline{p}^{N+1}_Y\left(l_1+n-\frac{1}{2},\cdots,l_n+n-\frac{1}{2}\right)=p^{N+1}_Y\left(\lambda_1,\cdots,\lambda_n\right)$. 
Thus, by linearity and \eqref{eq: 9},
\begin{multline*}
      \langle Y\rangle_{\frac{t}{2}}=\sum_{\lambda}\text{Prob}\left(x_k^{n,\frac{1}{2}}(t)=\lambda_k+n-k\right)\frac{\langle Y\rangle_{\chi^\lambda}}{\text{dim} \lambda} \\
    =\sum_{\lambda}\text{Prob}\left(x_k^{n,\frac{1}{2}}(t)=\lambda_k+n-k, 1\le k\le n\right)p^{N+1}_Y(\lambda_1,\cdots,\lambda_n)\\
    =\sum_l \text{Prob}\left(x_k^{n,\frac{1}{2}}(t)=\lambda_k+n-k, 1\le k\le n\right)\overline{p}^{N+1}_Y\left(l_1+n-\frac{1}{2},\cdots,l_n+n-\frac{1}{2}\right)\\
    =\mathbb{E}\left[\overline{p}^{N+1}_Y\left(x_1^{n,\frac{1}{2}}(t),\cdots,x^{n,\frac{1}{2}}_n(t)\right)\right].
\end{multline*}

Thus,
\begin{equation*}
     \langle Q^{n,\frac{1}{2}}_tY\rangle_{\frac{s}{2}}=\mathbb{E}\left[Q^{n,\frac{1}{2}}_t\overline{p}^{N+1}_Y\left(x_1^{n,\frac{1}{2}}(s),\cdots,x^{n,\frac{1}{2}}_n(s)\right)\right]
    =\mathbb{E}\left[\overline{p}^{N+1}_Y\left(x_1^{n,\frac{1}{2}}(t+s),\cdots,x^{n,\frac{1}{2}}_n(t+s)\right)\right]=\langle Y\rangle_{\frac{s+t}{2}}.
\end{equation*}

Similarly, when $N+1=2n$, for any  $O\in SO(N+1)$, \begin{equation*}
      e^{\frac{1}{2}t\text{Tr}(O-I)}=\sum_{\lambda}\text{Prob}\left(x_k^{n,-\frac{1}{2}}(t)=\lambda_k+n-k, 1\le k\le n\right)\frac{\chi^\lambda_{SO(2n)}(O)+\chi^{\lambda^*}_{SO(2n)}(O)}{2\text{dim}_{SO(2n)} \lambda}
\end{equation*}
by \eqref{eq: 8}.  
\begin{multline*}
     \langle Y\rangle_{\frac{t}{2}}=\sum_{\lambda}\text{Prob}\left(x_k^{n,-\frac{1}{2}}(t)=\lambda_k+n-k, 1\le k\le n\right)\frac{\langle Y\rangle_{\chi^\lambda}+\langle Y\rangle_{\chi^{\lambda^*}}}{2\text{dim}\lambda} \\
    =\sum_{\lambda}\text{Prob}\left(x_k^{n,-\frac{1}{2}}(t)=\lambda_k+n-k, 1\le k\le n\right)\frac{p^{N+1}_Y(\lambda_1,\cdots,\lambda_n)+p^{N+1}_Y(\lambda_1,\cdots,-\lambda_n)}{2}.
\end{multline*}

 Since $l_k=\lambda_k-k+1$ when $N+1=2n$, we define $\overline{p}^{N+1}_Y$ for even $N+1$ such that
\begin{equation*}
    \overline{p}^{N+1}_Y(l_1+n-1,\cdots,l_n+n-1)=\frac{p^{N+1}_Y(\lambda_1,\cdots,\lambda_n)+p^{N+1}_Y(\lambda_1,\cdots,-\lambda_n)}{2}.
\end{equation*}
Then,
\begin{multline*}
      \langle Y\rangle_{\frac{t}{2}} =\sum_l \text{Prob}\left(x_k^{n,-\frac{1}{2}}(t)=\lambda_k+n-k, 1\le k\le n\right)\overline{p}^{N+1}_Y(l_1+n-1,\cdots,l_n+n-1)\\
   =\mathbb{E}\left[\overline{p}^{N+1}_Y\left(x_1^{n,-\frac{1}{2}}(t),\cdots,x^{n,-\frac{1}{2}}_n(t)\right)\right].
\end{multline*}
Again, we have  $\langle Q^{n,-\frac{1}{2}}_tY\rangle_{_{\frac{s}{2}}}=\langle Y\rangle_{_{\frac{s+t}{2}}}$.
\end{proof}

\begin{Theorem}\label{Thm: 6}
$P_{\frac{t}{2}}X=Q^{n,a}_{t}X$ for all $X\in Z(U(\mathfrak{so}_{N+1}))$. In particular, $P_{\frac{t}{2}}$ is the Markov operator of the process $\left(x_1^{n,a}(t)>\cdots>x_n^{n,a}(t)\right)$.
\end{Theorem}
\begin{proof}
From Theorem \ref{Thm: 1} and Proposition \ref{Prop: 4} we have $\langle P_{\frac{t}{2}}Y\rangle_{_{\frac{s}{2}}}=\langle Q^{n,a}_tY\rangle_{_{\frac{s}{2}}}=\langle Y\rangle_{_{\frac{s+t}{2}}}$. The rest of proof follows the standard arguments as that of Proposition 4.4 in \cite{kuan2014threedimensional}. We only sketch the proof here. To show $P_{\frac{t}{2}}X=Q^{n,a}_{t}X$ for all $X\in Z(U(\mathfrak{so}_{N+1}))$, we show that if there exists a $Y\in Z(U(\mathfrak{so}_{N+1}))$ such that $\langle Y\rangle_t=0$ for all $t\ge 0$, then $Y=0$ by a contradiction.
\end{proof}

When restricting our non–commutative random walk to the Gelfand–Tsetlin subalgebra, which is the subalgebra of $U(\mathfrak{so}_{N+1})$ generated by the centres $Z(U(\mathfrak{so}_k))$, $1\le k\le N+1$,  it also matches the two–dimensional particle system along space–like paths:
\begin{Theorem}\label{Thm: 3}
Suppose $Y_1\in Z(U(\mathfrak{so}_{N_1+1})),\ldots,Y_r\in Z(U(\mathfrak{so}_{N_r+1}))$ are mapped to the shifted polynomials $\overline{p}^{N_1+1}_{Y_1},\ldots,\overline{p}^{N_r+1}_{Y_r}$ as in Proposition \ref{Prop: 4}  under the Harish-Chandra isomorphism. Assume  $N_1\ge\cdots\ge N_r$ and $t_1\le\ldots\le t_r$, then
\begin{equation*}
  \left \langle Y_1\left(P_{\frac{t_2-t_1}{2}}Y_2\right)\cdots \left(P_{\frac{tr-t1}{2}}Y_r\right) \right\rangle_{\frac{t_1}{2}} =\mathbb{E}\left[\overline{p}^{N_1+1}_{Y_1}\left(\mathbf{x}^{n_1,a_1}(t_1)\right)\cdots\overline{p}^{N_r+1}_{Y_r}\left(\mathbf{x}^{n_r,a_r}(t_r)\right)\right],
\end{equation*}
where $\mathbf{x}^{n_r,a_r}$ is the vector of $x^{n_r,a_r}_j$ with $1\le j\le n_r$.
\end{Theorem}
\begin{proof}
The proof only needs the Gibbs property of the Markov process  $ \left\{x_k^{n,a}|\ 1\le k\le n,\ a\in\left\{\pm 1/2\right\},\ n\ge 1\right\}$ (see e.g. \cite{borodin2011random}) and follows the same lines as that of Theorem 4.5 in \cite{kuan2014threedimensional}.
\end{proof}

\section{Three dimensional Gaussian fluctuations}
In this section, we show that certain elements  of the Gelfand–Tsetlin subalgebra are asymptotically Gaussian with an explicit covariance along space-like paths and time-like paths.

 For a Laurent polynomial $p(v)$, let $p(v)[v^r]$ denotes the coefficient of $v^r$ in $p(v)$. Then, the main theorem is 
\begin{Theorem}\label{Thm: 4}
Suppose $N_j=\lfloor \eta_j L\rfloor$, $t_j=\tau_j L$ for $1\le j\le r$. Assume $\min\{\tau_1,\ldots,\tau_r\}=\tau_1$. Then as $L\xrightarrow{}\infty$,
\begin{equation*}
    \left(\frac{\Phi_{2k_1}^{N_1+1}-\left\langle\Phi_{2k_1}^{N_1+1}\right\rangle_{\frac{t_1}{2}}}{2L^{2k_1}},\ldots,\frac{P_{\frac{t_r-t_1}{2}}\Phi_{2k_r}^{N_r+1}-\left\langle P_{\frac{t_r-t_1}{2}}\Phi_{2k_r}^{N_r+1}\right\rangle_{\frac{t_1}{2}}}{2L^{2k_r}} \right)\xrightarrow{} (\xi_1,\ldots,\xi_r),
\end{equation*}
where the convergence is with respect to the state $\langle\cdot\rangle_{\frac{t_1}{2}}$ and $(\xi_1,\ldots,\xi_r)$ is a Gaussian vector with covariance
\begin{equation*}
   \mathbb{E}[\xi_i\xi_j]=\left\{\begin{array}{lc}
    \frac{1}{(2\pi i)^2}\iint\displaylimits_{|v|>|u|} \left(\frac{(v+2)(\eta_i/2+\tau_i v)^2}{v}\right)^{k_i}\left(\frac{(u+2)(\eta_j/2+\tau_j u)^2}{u}\right)^{k_j}\frac{1}{(v-u)^2}dvdu,    & \eta_i\ge \eta_j,\tau_i\le \tau_j,  \\
       \\
 \frac{\sum_{l=1}^{k_j}c_{k_j,l} (\tau_j,\tau_i,\eta_j)}{(2\pi i)^2} \iint\displaylimits_{|v|<|u|} \left(\frac{(v+2)(\eta_i/2+\tau_i v)^2}{v}\right)^{k_i}\left(\frac{(u+2)(\eta_j/2+\tau_i u)^2}{u}\right)^{l}\frac{1}{(v-u)^2}dvdu,    & \eta_i< \eta_j,\tau_i\le \tau_j, 
    \end{array}\right.
\end{equation*}
where $c_{k,l}$ is defined as the coefficient in the following expansion: for $r\le -1$,
\begin{equation*}
    \sum_{l=1}^k c_{k,l}(\tau_2,\tau_1,\eta_2) \left(\frac{(v+2)\left(\frac{\eta_2}{2}+\tau_1v\right)^2}{v}\right)^l[v^r]=\left(\frac{(v+2)\left(\frac{\eta_2}{2}+\tau_2v\right)^2}{v}\right)^k[v^r].
\end{equation*}
\end{Theorem}
\begin{Remark}
It is straightforward from
inspection that the covariance for the random surface growth  when $ \eta_i\ge \eta_j,\tau_i\le \tau_j$ is different form the covariance for the spectra of overlapping stochastic Wishart matrices in  \cite{kuan2021threedimensional} due to the lack of a second integral. 

It is not obvious that the covariances along time-like paths are different. Heuristically, if they were the same, the covariance formula in Theorem \ref{Thm: 4}  when $\eta_i< \eta_j,\tau_i\le \tau_j $ could be written as 
\begin{multline*}
     \frac{C_1}{(2\pi i)^2}\iint\displaylimits_{|v|<|u|} \left(\frac{(v+2)(\eta_i/2+\tau_i v)^2}{v}\right)^{k_i}\left(\frac{(u+2)(\eta_j/2+\tau_j u)^2}{u}\right)^{k_j}\frac{1}{(v-u)^2}dvdu\\
     +C_2\cdot\text{Residue}\left(\left(\frac{(v+2)(\eta_i/2+\tau_i v)^2}{v}\right)^{k_i}\right)\text{Residue}\left(\left(\frac{(u+2)(\eta_j/2+\tau_j u)^2}{u}\right)^{k_j}\right),
\end{multline*}
where $C_1$ and $C_2$ are constants which are  independent of $k_{i}$ and $k_j$. However, after checking a few examples, the constants $C_1$ and $C_2$ do not exist. For example, letting $k_i=1,2$ and $k_j=1$ and solving for $C_{1,2}$ yields $C_1=\tau_i^2/\tau_j^2$ and $C_2=2\tau_1(\tau_2-\tau_1)/(\eta_2\tau_2)$. However, if  $C_1=\tau_i^2/\tau_j^2$ and $C_2=2\tau_1(\tau_2-\tau_1)/(\eta_2\tau_2)$,  the covariances  are not equal when $k_i=1$ and $k_j=2$.
\end{Remark}

\subsection{Gaussian  fluctuations along space-like paths}
In this section, we focus on the space-like paths. It was proved that the random surface growth converges to a deterministic limit shape and the fluctuations around the limit shape are described by the Gaussian free field fluctuations  (see e.g. \cite{borodin2011random,10.1214/EJP.v19-3732}).

Let $G(z)=G(\nu,\eta,\tau;z)$ be the function
\begin{equation*}
  G(\nu,\eta,\tau;z)=\tau\frac{z+z^{-1}}{2}+\eta\log\left(\frac{z+z^{-1}}{2}-1\right)-\nu\log z, 
\end{equation*}
and $\mathcal{D}$ be the connected domain consisting of all triples $(\nu,\eta,\tau)$ such that $ G(\nu,\eta,\tau;z)$ has a unique critical point in the region $\mathbb{H-D}=\left\{z|\Im z>0,\ |z|>1\right\}$. Let $\Upsilon$ be the map sending $(\nu,\eta,\tau)\in\mathcal{D}$ to the critical point of $ G(\nu,\eta,\tau;z)$ in $  \mathbb{H-D}$. Specifically, 
\begin{equation*}
    \mathcal{D}=\left\{(\nu,\eta,\tau)| l(\eta,\tau)< \nu< r(\eta,\tau), \tau,\eta>0\right\},
\end{equation*}
where 
\begin{equation*}
    \begin{array}{l}
        r(\eta,\tau)=\sqrt{-\frac{\tau^2}{2}+5\tau\eta+1+\frac{\tau^2}{2}\left(1+\frac{4\eta}{\tau}\right)^{3/2}},\\
        \\
    l(\eta,\tau)=\left\{\begin{array}{lc}
      0, & \frac{\eta}{\tau}\le 2,\\
      \sqrt{-\frac{\tau^2}{2}+5\tau\eta+1-\frac{\tau^2}{2}\left(1+\frac{4\eta}{\tau}\right)^{3/2}}, & \frac{\eta}{\tau}> 2. 
    \end{array}\right. 
\end{array}
\end{equation*}

Let $H(x,N,t)$ be the height function of $\mathbf{x}$, which is defined as the number of particles to the right of $(x,n,a)$ at time $t$.  We recall the limit shape of the height function $H$ from \cite{borodin2011random}.
\begin{Theorem}\cite{borodin2011random}\label{Thm: 5}
 For any $\left(\nu,\frac{\eta}{2},\tau\right)\in\mathcal{D}$, suppose $\Upsilon\left(\nu,\frac{\eta}{2},\tau\right)=z_0$, then the limit shape exists,
\begin{equation}\label{eq: 10}
h(z_0)=    \lim_{L\xrightarrow[]{}\infty}\frac{1}{L}\mathbb{E}H(\lfloor\nu L\rfloor, \lfloor \eta L\rfloor, \tau L)=\Im\left(\frac{G\left(z_0\right)}{\pi}    \right).
\end{equation}
\end{Theorem}

By Theorem 1.1 in \cite{10.1214/EJP.v19-3732}, the height fluctuations converge to a Gaussian free field on the domain $\mathcal{D}$. The proof is based on the fact that the interacting particle system is a determinantal point process. Analogous to \cite{Borodin_2007,10.1214/EJP.v19-3732}, the determinantal structure derived in Theorem \ref{thm: 1} also leads to the convergence of the moments of height fluctuations to that of a Gaussian free field along space-like paths. As a result, we could  generalize Theorem 1.1 in \cite{10.1214/EJP.v19-3732} as the  following.
\begin{Theorem}\label{Thm: 2}
For any $r\in\mathbb{N}^+$, let $\varkappa_j=\left(\nu_j,\frac{\eta_j}{2},\tau_j\right)\in\mathcal{D}$ for $1\le j\le r$. Define
\begin{equation*}
    H_L(\nu,\eta,\tau)=\sqrt{\pi}\left(H(\nu L,\lfloor\eta L\rfloor,\tau L)-\mathbb{E}H(\nu L,\lfloor\eta L\rfloor,\tau L)\right)
\end{equation*}
and $\Upsilon_j=\Upsilon(\varkappa_j)$. Assume $\{\varkappa_j\}_{j=1}^r$ lie on a space-like path, that is  $\eta_1\ge \ldots\ge \eta_r$ and $\tau_1\le \ldots\le \tau_r$, then
\begin{equation*}
    \lim_{L\xrightarrow[]{}\infty}\mathbb{E}(H_L(\varkappa_1)\cdots H_L(\varkappa_r))=\left\{ \begin{array}{lr}
         \sum_{\sigma}\prod_{i=1}^{r/2} \mathcal{G}(\Upsilon_{\sigma(2i-1)},\Upsilon_{\sigma(2i)}),  & r\ \text{even},  \\
            0, & r\ \text{odd},
             \end{array}\right. 
\end{equation*}
where the sum is over all fixed point free involutions $\sigma$ on $\{1,\ldots,r\}$ and $\mathcal{G}$ is the funciton
\begin{equation*}
      \mathcal{G}(z,w)=\frac{1}{2\pi}\log \left|\frac{z+z^{-1}-\overline{w}-\overline{w}^{-1}}{z+z^{-1}-w-w^{-1}}\right|.
\end{equation*}
\end{Theorem}

Let $p_{2k}^{N+1}=\sum_{i=1}^n l_i^{2k}$, which is the image of $\frac{\Phi^{N+1}_{2k}}{2}$ under the Harish-Chandra isomorphism. We first relate the height function and the polynomials $p_{2k}^{N+1}$ as in \cite{BB-Plancherel}. The definition of the height function implies that
\begin{equation*}
    \frac{d}{du}H(u,N)=-\sqrt{\pi}\sum_{s=1}^{n}\delta(u-(\lambda_s^{N}-s+n)).
\end{equation*}
Recall that for $N+1=2n+1/2+a$,  $l_i=\lambda_i-i+\frac{3}{4}-\frac{a}{2}$,
\begin{equation}
    \frac{d}{du}H(u,N)=-\sqrt{\pi}\sum_{s=1}^{n}\delta\left(u-\left(l_s+n-\frac{3}{4}+\frac{a}{2}\right)\right).
\end{equation}

Then, let $u=Lx$, $N=\lfloor L\eta\rfloor$, $t=L\tau$,
\begin{multline*}
     p_{2k}^{N+1}=\int_{0}^\infty \left(u-n+\frac{3}{4}-\frac{a}{2}\right)^{2k}\left(\sum_{s=1}^{n} \delta\left(u-n-l_s+\frac{3}{4}-\frac{a}{2}\right)\right)du
     \\=-\frac{1}{\sqrt{\pi}}\int_{0}^\infty \left(u-n+\frac{3}{4}-\frac{a}{2}\right)^{2k} \frac{d}{du}H(u,N)du\\
     \approx \frac{1}{\sqrt{\pi}}\left(n^{2k+1}+2k\int_{0}^\infty (u-n)^{2k-1} H(u,N)du\right).  \end{multline*}
   
The relation between height function and the polynomial $p_{2k}^{N+1}$ indicates the convergence to a Gaussian free field.
  \begin{Proposition}\label{Prop: 7}
  Suppose $N_j=\lfloor \eta_j L\rfloor$, $t_j=\tau_j L$ for $1\le j\le r$. Assume $\eta_1\ge \ldots\ge \eta_r$ and $\tau_1\le \ldots\le \tau_r$, then as $L\xrightarrow{}\infty$,
\begin{equation*}
    \left(\frac{\Phi_{2k_1}^{N_1+1}-\left\langle\Phi_{2k_1}^{N_1+1}\right\rangle_{\frac{t_1}{2}}}{2L^{2k_1}},\ldots,\frac{P_{\frac{t_r-t_1}{2}}\Phi_{2k_r}^{N_r+1}-\left\langle P_{\frac{t_r-t_1}{2}}\Phi_{2k_r}^{N_r+1}\right\rangle_{\frac{t_1}{2}}}{2L^{2k_r}} \right)\xrightarrow{} (\xi_1,\ldots,\xi_r),
\end{equation*}
where the convergence is with respect to the state $\langle\cdot\rangle_{\frac{t_1}{2}}$ and $(\xi_1,\ldots,\xi_r)$ is a Gaussian vector. 
 \end{Proposition}
 \begin{proof}
 The convergence to a Gaussian vector (along both space-like and time-like paths)  for elements in $Z(U(\mathfrak{so}_{N+1}))$ can be proved using combinatorial arguments analogous to Proposition 5.1 of \cite{BB-Plancherel}. 
 Another way to see that is to apply  Theorem \ref{Thm: 3} and  Theorem \ref{Thm: 2} to show that the limit as $L\xrightarrow[]{}\infty$ of joint moments of 
 \begin{equation*}
     \frac{\Phi_{2k_j}^{N_j+1}-\left\langle\Phi_{2k_j}^{N_j+1}\right\rangle_{\frac{t_j}{2}}}{2L^{2k_j}}
 \end{equation*} satisfy the Wick's formula, which implies the convergence to a Gaussian vector.
 \end{proof}

Next, we calculate the covariance structure for elements in $Z(U(\mathfrak{so}_{N+1}))$.
\begin{Proposition}\label{prop: 6}
Suppose $\eta_1\ge \eta_2$ and $\tau_1\le \tau_2$,
\begin{multline}\label{eq: 7.1}
  \lim_{L\xrightarrow{}\infty}   \left\langle\frac{\Phi_{2k}^{\lfloor\eta_1 L\rfloor+1}-\left\langle\Phi_{2k}^{\lfloor\eta_1 L\rfloor+1}\right\rangle_{\tau_1L/2}}{2L^{2k}}\cdot\frac{P_{(\tau_2-\tau_1)/2L}\Phi_{2l}^{\lfloor\eta_2 L\rfloor+1}-\left\langle P_{(\tau_2-\tau_1)/2L}\Phi_{2l}^{\lfloor\eta_2 L\rfloor+1}\right\rangle_{\tau_1 L/2}}{2L^{2l}}\right\rangle_{\tau_1L/2}  \\
  = \frac{1}{(2\pi i)^2}\iint\displaylimits_{|v|>|u|} \left(\frac{(v+2)(\eta_1/2+\tau_1 v)^2}{v}\right)^k\left(\frac{(u+2)(\eta_2/2+\tau_2 u)^2}{u}\right)^l\frac{1}{(v-u)^2}dvdu.    
\end{multline}
\end{Proposition} 
\begin{proof}

 By Theorem \ref{Thm: 3}, it suffices to compute the limit of
\begin{equation*}
   \frac{\mathbb{E}\left[\left(   p_{2k}^{\lfloor L\eta_1\rfloor+1}(\tau_1 L)-\mathbb{E} p_{2k}^{\lfloor L\eta_1\rfloor+1}(\tau_1 L)\right)\left( p_{2l}^{\lfloor L\eta_2\rfloor+1}(\tau_2 L)-\mathbb{E} p_{2l}^{\lfloor L\eta_2\rfloor+1}(\tau_2 L)\right)\right]}{L^{2k+2l}}  
\end{equation*}
as $L\xrightarrow{}\infty$. Apply Theorem \ref{Thm: 2} and dominated convergence theorem,
\begin{multline}\label{eq: 4}
     \lim_{L\xrightarrow{}\infty} \frac{\mathbb{E}\left[\left(   p_{2k}^{\lfloor L\eta_1\rfloor+1}(\tau_1 L)-\mathbb{E} p_{2k}^{\lfloor L\eta_1\rfloor+1}(\tau_1 L)\right)\left( p_{2l}^{\lfloor L\eta_2\rfloor+1}(\tau_2 L)-\mathbb{E} p_{2l}^{\lfloor L\eta_2\rfloor+1}(\tau_2 L)\right)\right]}{L^{2k+2l}}  \\
 =\displaystyle{\frac{4kl}{\pi}\int\displaylimits_{\left(x,\frac{\eta_1}{2},\tau_1\right)\in\mathcal{D}}\int\displaylimits_{\left(y,\frac{\eta_2}{2},\tau_2\right)\in\mathcal{D}} \left(x-\frac{\eta_1}{2}\right)^{2k-1}\left(y-\frac{\eta_2}{2}\right)^{2l-1}\mathcal{G}\left(\Upsilon\left(y,\frac{\eta_2}{2},\tau_2\right),\Upsilon\left(x,\frac{\eta_1}{2},\tau_1\right)\right)dxdy.}  
\end{multline}

Recall that  
\begin{equation*}
   \mathcal{G}(z,w)=\frac{1}{2\pi}\log \left|\frac{z+z^{-1}-\overline{w}-\overline{w}^{-1}}{z+z^{-1}-w-w^{-1}}\right|
\end{equation*}and
$\Upsilon:\mathcal{D}\xrightarrow{}\mathbb{H-D}$ is defined by sending $(d,l,t)\in\mathcal{D}$ to the solution of 
\begin{equation*}
   \left\{\begin{array}{cc}
    z+\overline{z}-\frac{1}{z\overline{z}}=-\frac{2l-2d-t}{t},      
    \\
    \frac{1}{z}+\frac{1}{\overline{z}}-z\overline{z}=-\frac{2l+2d-t}{t}. 
    \end{array}\right.
\end{equation*}
Now, let $z(x,\eta_1,\tau_1)=\Upsilon(x,\eta_1/2,\tau_1)$, $w(y,\eta_2,\tau_2)=\Upsilon(y,\eta_2/2,\tau_2)$, then $x(z)=\eta_1+\frac{\tau_1}{2}\frac{z^3-z^2+(\frac{2\eta_1}{\tau_1}-1)z+1}{z(z-1)}$ and $y(w)=\eta_2+\frac{\tau_2}{2}\frac{w^3-w^2+(\frac{2\eta_2}{\tau_2}-1)w+1}{w(w-1)}$.

Using the symmetry of the integrand,  the right-hand-side of \eqref{eq: 4} can be written as 
\begin{equation*}
    \frac{4kl}{(2\pi i)^2}\int\displaylimits_{\gamma_1}\int\displaylimits_{\gamma_2} \left(x(z)-\frac{\eta_1}{2}\right)^{2k-1}\left(y(w)-\frac{\eta_2}{2}\right)^{2l-1}\log(z+z^{-1}-w-w^{-1})\frac{d(x(z))}{dz}\frac{d(y(w))}{dw}dzdw,
\end{equation*}
where for each fixed $\eta_i$ and $\tau_i$, the integration contour $\gamma_i$   is a positively orientated path given by all points in the set $\left\{z(x,\eta_i,\tau_i),\overline{z(x,\eta_i,\tau_i)}|\left(x,\frac{\eta_i}{2},\tau_i\right)\in \mathcal{D}\right\}$ (see figure \ref{fig:1}). 
\begin{figure}
    \centering
    \includegraphics[width=0.5\columnwidth]{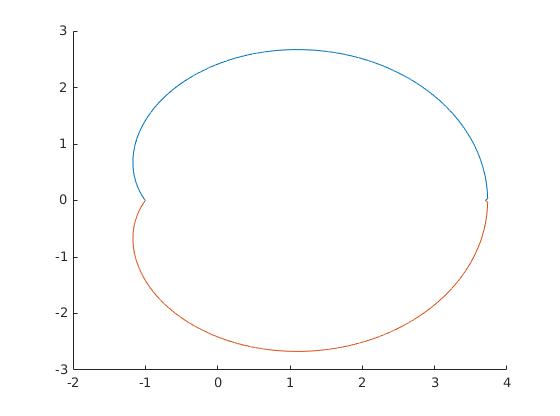}
    \caption{Integration contour $\gamma_i$ when $\eta_i/\tau_i=4$}
    \label{fig:1}
\end{figure}

Integrate by part in $z$ and $w$ and let $v=\frac{z+z^{-1}}{2}-1$, $u=\frac{w+w^{-1}}{2}-1$, we have 
\begin{equation}\label{eq: 7.3}
     \frac{1}{(2\pi i)^2}\oint\displaylimits_{\gamma_1'}\oint\displaylimits_{\gamma_2'} \left(\frac{(v+2)(\eta_1/2+\tau_1 v)^2}{v}\right)^k\left(\frac{(u+2)(\eta_2/2+\tau_2 u)^2}{u}\right)^l\frac{1}{(v-u)^2}dvdu.  
\end{equation}
Here $\gamma_1'$ and $\gamma_2'$ are images of $\gamma_1$ and $\gamma_2$ under the change of variable. Last, modify  $\gamma_1'$ and $\gamma_2'$ to two concentric circles centered at $0$, then the double integral could be written as 
\begin{equation*}
     \frac{1}{(2\pi i)^2}\iint\displaylimits_{|v|>|u|} \left(\frac{(v+2)(\eta_1/2+\tau_1 v)^2}{v}\right)^k\left(\frac{(u+2)(\eta_2/2+\tau_2 u)^2}{u}\right)^l\frac{1}{(v-u)^2}dvdu.  
\end{equation*}

\end{proof}

Note that
\begin{equation*}
 \frac{(v+2)(\eta_1/2+\tau_1v)^2}{v}=  \frac{\eta_1^2}{2v}+\tau_1^2v^2+(2\tau_1^2+\eta_1\tau_1)v+\frac{\eta_1^2}{4}+2\eta_1\tau_1,
\end{equation*}
then the double integral can be computed using residue theorem and the Taylor series
\begin{equation*}
    (v-u)^{-2}=v^{-2}\left(1+2\frac{u}{v}+3\frac{u^2}{v^2}+\cdots\right).
\end{equation*}

\begin{Example}
Recall $ \Phi_2^{N+1}=2\sum_{m=1}^n\left((F_{mm}+\rho_m)^2+2\sum_{-m<i<m}F_{mi}F_{im}\right)$.
When $k=l=1$, \eqref{eq: 7.1} equals $\tau_1^2\eta_2^2+1/2\tau_1\eta_1\eta_2^2$.

By direct computation,
\begin{equation*}
   P_{(\tau_2-\tau_1)L}\Phi_2^{N+1}=\Phi_2^{N+1} +2\sum_{m=1}^{n}\left(2(\tau_2-\tau_1)L+\sum_{-m<i<m}2(\tau_2-\tau_1)L\right)
\end{equation*}
and
\begin{equation*}
    \left\langle\Phi_2^{N+1}\right\rangle_{\tau_1L}=2\sum_{m=1}^{n}\left(\left(2\tau_1L+\rho_m^2\right)+2\sum_{-m<i<m}2\tau_1L\right),
\end{equation*}
thus 
\begin{equation*}
    P_{(\tau_2-\tau_1)L}\Phi_2^{N+1}-\left\langle P_{(\tau_2-\tau_1)L}\Phi_2^{N+1}\right\rangle_{\tau_1L}=\Phi_2^{N+1}-\left\langle\Phi_2^{N+1}\right\rangle_{\tau_1L},
\end{equation*}
moreover,
\begin{multline*}
   \lim_{L\xrightarrow{}\infty}   \left\langle\frac{\Phi_{2}^{\lfloor\eta_1 L\rfloor+1}-\left\langle\Phi_{2}^{\lfloor\eta_1 L\rfloor+1}\right\rangle_{\tau_1L/2}}{2L^{2k}}\cdot\frac{P_{(\tau_2-\tau_1)/2L}\Phi_{2}^{\lfloor\eta_2 L\rfloor+1}-\left\langle P_{(\tau_2-\tau_1)/2L}\Phi_{2}^{\lfloor\eta_2 L\rfloor+1}\right\rangle_{\tau_1 L/2}}{2L^{2l}}\right\rangle_{\tau_1L/2} \\ = \lim_{L\xrightarrow{} \infty}\frac{1}{L^4} \left\langle\left(\frac{\Phi_2^{\lfloor\eta_1 L\rfloor+1}}{2}-\left\langle\frac{\Phi_2^{\lfloor\eta_1 L\rfloor+1}}{2}\right\rangle_{\frac{\tau_1L}{2}}\right)\cdot\left(\frac{\Phi_2^{\lfloor\eta_2 L\rfloor+1}}{2}-\left\langle\frac{\Phi_2^{\lfloor\eta_2 L\rfloor+1}}{2}\right\rangle_{\frac{\tau_1L}{2}}\right)\right\rangle_{\frac{\tau_1L}{2}} =
\tau_1^2\eta_2^2+1/2\tau_1\eta_1\eta_2^2,  
\end{multline*}

which matches \eqref{eq: 7.1}.

\end{Example}

\subsection{Gaussian fluctuations along time-like paths}

Next, we prove an analogy of Proposition 5.3 in \cite{kuan2014threedimensional}, which provides an  expression for the covariance along the time-like paths. 

Given a partition $\rho=(\rho_1,\ldots,\rho_l)$, let $\Phi_\rho=\prod_{i=1}^l\Phi_{2\rho_i}$, denote $|\rho|=2(\rho_1+\cdots+\rho_l)$ and $\text{wt}(\rho)= |\rho|+l$.

\begin{Proposition}\label{Prop: 8}
 Let $\eta,\tau>0$, set $N=\lfloor\eta L\rfloor$, and $t=\tau L$, then
\begin{enumerate}[{(1)}]
    \item $<\Phi_{\rho}^{N+1}>_{t}=\Theta(L^{\text{wt}(\rho)})$ .
    \item There exist constants $c'_{k,\rho}(\tau,\eta)$ such that
\begin{equation*}
    P_{\tau L} \Phi_{2k}^{N+1}=\sum_\rho \left(c'_{k,\rho}(\tau,\eta)+o(1)\right)L^{2k+1-\text{wt}(\rho)}\Phi_{\rho}^{N+1},
\end{equation*}
where the sum is over $\rho$ such that $\text{wt}(\rho)\le 2k+1$.

\item For any $\tau_2\ge\tau_1$, there exist constants $c_{k,j}(\tau_2,\tau_1,\eta_2)$ such that 
\begin{multline*}
  \lim_{L\xrightarrow{}\infty}   \left<\frac{\Phi_{2m}^{\lfloor\eta_1 L\rfloor+1}-\left<\Phi_{2m}^{\lfloor\eta_1 L\rfloor+1}\right>_{\tau_1L}}{L^{2m}}\cdot\frac{P_{(\tau_2-\tau_1)/2L}\Phi_{2k}^{\lfloor\eta_2 L\rfloor+1}-\left<P_{(\tau_2-\tau_1)L}\Phi_{2k}^{\lfloor\eta_2 L\rfloor+1}\right>_{\tau_1 L}}{L^{2k}}\right>_{\tau_1L}\\
   =   \lim_{L\xrightarrow{}\infty}   \sum_{j=1}^k c_{k,j}(\tau_2,\tau_1,\eta_2) \left\langle\frac{\Phi_{2m}^{\lfloor\eta_1 L\rfloor+1}-\left\langle\Phi_{2m}^{\lfloor\eta_1 L\rfloor+1}\right\rangle_{\tau_1L}}{L^{2m}}\cdot\frac{\Phi_{2j}^{\lfloor\eta_2 L\rfloor+1}-\left\langle\Phi_{2j}^{\lfloor\eta_2 L\rfloor+1}\right\rangle_{\tau_1L}}{L^{2j}}\right\rangle_{\tau_1L}.  
\end{multline*}

\end{enumerate}

\end{Proposition}

\begin{proof}
We can prove all the results following the same arguments in \cite{kuan2014threedimensional}. Another intuitive way to see that (1) is true is to use Proposition \ref{Prop: 4}, Theorem \ref{Thm: 5} and Theorem \ref{Thm: 2}.  
 Recall the limit of height function in frozen region is given by $(\frac{\eta}{2}-x)$, and in liquid region $\mathcal{D}$ is expected to be the function  $h\left(\Upsilon\left(x,\frac{\eta}{2},\tau\right)\right)$ defined in \eqref{eq: 10}. Note that $l\left(\frac{\eta}{2},\tau\right)<\frac{\eta}{2}< r\left(\frac{\eta}{2},\tau\right)$ when $\eta>0$. Let $N=\lfloor\eta L\rfloor$, $t=\tau L$, then interchange the expectation and integration, we have
\begin{multline}\label{eq: 7.2}
  \lim_{L\xrightarrow{}\infty} \left\langle\frac{\Phi^{N+1}_{2k}}{2L^{2k+1}}\right\rangle_{\frac{t}{2}} =   \lim_{L\xrightarrow{}\infty}\frac{\mathbb {E}[p_{2k}^{N+1}]}{L^{2k+1}}
   =\frac{1}{\sqrt{\pi}}\bigg(2k \int_{l\left(\frac{\eta}{2},\tau\right)}^{r\left(\frac{\eta}{2},\tau\right)}\left(x-\frac{\eta}{2}\right)^{2k-1}h\left(\Upsilon\left(x,\frac{\eta}{2},\tau\right)\right)dx\\
   -2k \int_{0}^{l\left(\frac{\eta}{2},\tau\right)}\left(x-\frac{\eta}{2}\right)^{2k}dx
   +\left(\frac{\eta}{2}\right)^{2k+1}\bigg).
\end{multline}

The right-hand-side of \eqref{eq: 7.2} is strictly bounded below by
\begin{multline*}
     \frac{1}{\sqrt{\pi}}\left(2k \int_{l\left(\frac{\eta}{2},\tau\right)}^{\frac{\eta}{2}}\left(x-\frac{\eta}{2}\right)^{2k-1}\left(\frac{\eta}{2}-l\left(\frac{\eta}{2},\tau\right)\right)dx
   -2k \int_{0}^{l\left(\frac{\eta}{2},\tau\right)}\left(x-\frac{\eta}{2}\right)^{2k}dx
   +\left(\frac{\eta}{2}\right)^{2k+1}\right)\\
   =\frac{\left(\frac{\eta}{2}\right)^{2k+1}-\left(\frac{\eta}{2}-l\left(\frac{\eta}{2},\tau\right)\right)^{2k+1}}{\sqrt{\pi}(2k+1)}.
\end{multline*}
Thus, when $\eta>0$, the right-hand side of \eqref{eq: 7.2} is strictly larger than $0$ and is finite, which means $\left\langle\Phi^{N+1}_{2k}\right\rangle_{t}=\Theta(L^{2k+1})$. Also by Theorem \ref{Thm: 4}, we know that 
\begin{equation*}
     \left\langle\left(\Phi^{N+1}_{2k}\right)^2\right\rangle_{t}=\Theta(L^{2(2k+1)}).
\end{equation*}
Thus by Cauchy-Schwartz inequality, $\langle\Phi_\rho\rangle_{t}=\mathcal{O}(L^{\text{wt}(\rho)})$.

To get a lower bound, one notice that  $\lim_{L\xrightarrow{}\infty}\frac{p_{2k}^{N+1}(\lambda(t))}{L^{2k+1}}>0$ a.e. with respect to $\mathbb{E}$,
\begin{equation*}
     \lim_{L\xrightarrow{}\infty}\left\langle\prod_{i=1}^l \frac{\Phi^{N+1}_{2\rho_i}}{2L^{2\rho_i+1}}\right\rangle_{t/2} =\lim_{L\xrightarrow{}\infty} \mathbb{E}\left(\prod_{i=1}^l \frac{p^{N+1}_{2k_i}}{L^{2\rho_i+1}}\right)>0.
\end{equation*}

(2) follows from the fact that   $\langle P_{\tau_1 L} \Phi_{2k}^{N+1}\rangle_{\tau_2 L}=\Theta(L^{\text{2k+1}})$, thus only $\text{wt}(\rho)\le 2k+1$ terms have nonzero coefficients.

To prove (3), we first apply (2) to the left-hand-side. Note that $\frac{\Phi_{2k}^{N+1}-   \langle\Phi_{2k}^{N+1}\rangle_{t}}{L^{2k}}
$ corresponds to the moment of the fluctuation of height function, thus converges to a Gaussian random variable. Heuristically, $\Phi_{2k}^{N+1}\approx \langle\Phi_{2k}^{N+1}\rangle_{t}+\xi L^{2k}$, where $\xi$ is a Gaussian random variable.
\begin{multline*}
    \prod_{i=1}^l\Phi_{2\rho_i}-\left\langle\prod_{i=1}^l\Phi_{2\rho_i}\right\rangle_{\tau L}
    =\sum_{j=1}^l\left\langle\Phi_{2\rho_1}\right\rangle_{\tau L}\cdots\widehat{\left\langle\Phi_{2\rho_j}\right\rangle_{\tau L}}\cdots\left\langle\Phi_{2\rho_l}\right\rangle_{\tau L}\left(\Phi_{2\rho_j}-\left\langle\Phi_{2\rho_j}\right\rangle_{\tau L}\right)+\text{ smaller order terms}
\end{multline*}
Thus in the asymptotic limit, we could replace $P_t\Phi_\rho$ with a linear combination of $\Phi_{2k}$ with $2k+1\le \text{wt}(\rho)$.
\end{proof}

Expand $(v-u)^{-2}$ as $v^{-2}\left(1+2\frac{u}{v}+3\frac{u^2}{v^2}+\cdots\right)$ in Proposition \ref{prop: 6} and take residues, one obtains 
\begin{equation}\label{eq: 7.4}
    \sum_{j=1}^k c_{kj}(\tau_2,\tau_1,\eta_2) \left(\frac{(v+2)\left(\frac{\eta_2}{2}+\tau_1v\right)^2}{v}\right)^j[v^r]=\left(\frac{(v+2)\left(\frac{\eta_2}{2}+\tau_2v\right)^2}{v}\right)^k[v^r]
\end{equation}
for $r\le -1$.
We make use of the  expansion \eqref{eq: 7.4} to compute the coefficients $c_{k,j}$, which provides a formula for the covariance along the time-like paths.
\begin{Example}
When $k=2$, solving \eqref{eq: 7.4} for $r=-1,-2$, we have $c_{2,2}=1$ and $c_{2,1}=4\eta_2(\tau_2-\tau_1)$, which agrees with the expansion of $P_t\Phi_4^{N+1}$ derived in Example \ref{exp: 2}:
\begin{equation*}
    P_t\Phi_4^{N+1}=\Phi_4^{N+1}+16tn\Phi_2^{N+1}+\text{constant}.
\end{equation*}

\end{Example}
\begin{Example}
\begin{enumerate}[{(1)}]
    \item When $k=3$,\\ $c_{3,3}=1$,\\ $c_{3,2}=-6\eta_2(\tau_1-\tau_2)$,\\ $c_{3,1}=-\frac{3}{2}\left(\eta_2^3\tau_1-6\eta_2^2\tau_1^2-\eta_2^3\tau_2+16\eta_2^2\tau_1\tau_2-10\eta_2^2\tau_2^2\right)$.
    \item When $k=4$,\\ $c_{4,4}=1$,\\ $c_{4,3}=-8\eta_2(\tau_1-\tau_2)$,\\ $c_{4,2}=-2\left(\eta_2^3\tau_1-10\eta_2^2\tau_1^2-\eta_2^3\tau_2+24\eta_2^2\tau_1\tau_2-14\eta_2^2\tau_2^2\right)$,\\
    $c_{4,1}=-\frac{1}{2}\left(-\eta_2^5\tau_1+12\eta_2^4\tau_1^2-32\eta_2^3\tau_1^3+\eta_2^5\tau_2-40\eta_2^4\tau_1\tau_2+144\eta_2^3\tau_1^2\tau_2+28\eta_2^4\tau_2^2-224\eta_2^3\tau_1\tau_2^2+122\eta_2^3\tau_2^3\right)$.
\end{enumerate}

\end{Example}

Last, we briefly show the convergence to a Gaussian vector along time-like paths.
\begin{Proposition}
  Suppose $N_j=\lfloor \eta_j L\rfloor$, $t_j=\tau_j L$ for $1\le j\le r$. Assume $\eta_1\le \ldots\le \eta_r$ and $\tau_1\le \ldots\le \tau_r$, then as $L\xrightarrow{}\infty$,
\begin{equation*}
    \left(\frac{\Phi_{2k_1}^{N_1+1}-\left\langle\Phi_{2k_1}^{N_1+1}\right\rangle_{\frac{t_1}{2}}}{2L^{2k_1}},\ldots,\frac{P_{\frac{t_r-t_1}{2}}\Phi_{2k_r}^{N_r+1}-\left\langle P_{\frac{t_r-t_1}{2}}\Phi_{2k_r}^{N_r+1}\right\rangle_{\frac{t_1}{2}}}{2L^{2k_r}} \right)\xrightarrow{} (\xi_1,\ldots,\xi_r),
\end{equation*}
where the convergence is with respect to the state $\langle\cdot\rangle_{\frac{t_1}{2}}$ and $(\xi_1,\ldots,\xi_r)$ is a Gaussian vector. 
 \end{Proposition}
\begin{proof}
Same as the proof for space-like paths in Proposition \ref{Prop: 7}, we could  apply Theorem \ref{Thm: 3}, \ref{Thm: 2} and Proposition \ref{Prop: 8} to  derive the explicit formula for the joint moments, which  satisfies the Wick's formula and implies the convergence to a Gaussian vector.

\end{proof}

\end{document}